\documentclass[11pt,a4paper]{amsart}
\usepackage{amssymb,amsmath}
\usepackage[alphabetic]{amsrefs}
\usepackage{amsmath,amssymb}
\usepackage{mathrsfs}
\usepackage[dvipdfmx]{graphicx,color}
\usepackage{wrapfig}
\usepackage{calc}
\usepackage{extarrows}
\usepackage{amsthm}
\usepackage{latexsym}
\usepackage{slashed}
\usepackage{framed}
\usepackage{ascmac}
\usepackage[all]{xy}
\usepackage{hyperref}
\usepackage[alphabetic]{amsrefs}
\usepackage[mathscr]{eucal}

\makeatletter
\@addtoreset{equation}{section}
\makeatother

\DeclareMathAlphabet{\mathfrak}{U}{euf}{m}{n}
\SetMathAlphabet{\mathfrak}{bold}{U}{euf}{b}{n}

\theoremstyle{definition}
\newtheorem{defn}[equation]{Definition}
\theoremstyle{plain}
\newtheorem{thm}[equation]{Theorem}
\newtheorem{prp}[equation]{Proposition}
\newtheorem{lem}[equation]{Lemma}
\newtheorem{cor}[equation]{Corollary}

\theoremstyle{remark}
\newtheorem{remk}[equation]{Remark}
\newtheorem{exmp}[equation]{Example}
\newtheorem*{remk*}{Remark}
\newtheorem{blank}[equation]{}

\newcommand{\ma}[1]{\begin{align*} #1 \end{align*}}
\newcommand{\maa}[1]{\begin{align} #1 \end{align}}

\newcommand{\real}{\mathbb R}
\newcommand{\comp}{\mathbb C}
\newcommand{\zahl}{\mathbb Z}
\newcommand{\quot}{\mathbb Q}
\newcommand{\nat}{\mathbb N}

\DeclareMathOperator*{\Hom}{Hom}

\newcommand{\id}{\mathrm{id}}

\DeclareMathOperator*{\Ker}{Ker}

\newcommand{\bk} [1]{\left(  #1 \right)}
\newcommand{\ebk}[1]{\left<  #1 \right>}

\newcommand{\ssbk}[1]{\left\| #1 \right\|}

\newcommand{\pmx}[1]{\begin{pmatrix} #1 \end{pmatrix}}

\newcommand{\Cliff}{\mathbb{C}\ell }

\newcommand{\xra}{\xrightarrow}

\DefineSimpleKey{bib}{howpublished}

\renewcommand{\PrintDOI}[1]{%
  \href{http://dx.doi.org/#1}{{\tt DOI:#1}}%
  \IfEmptyBibField{volume}{, (to appear in print)}{}%
}

\renewcommand{\eprint}[1]{#1}

\BibSpec{book}{%
    +{}  {\PrintPrimary}                {transition}
    +{.} { \textit}                     {title}
    +{.} { }                            {part}
    +{:} { \textit}                     {subtitle}
    +{,} { \PrintEdition}               {edition}
    +{}  { \PrintEditorsB}              {editor}
    +{,} { \PrintTranslatorsC}          {translator}
    +{,} { \PrintContributions}         {contribution}
    +{,} { }                            {series}
    +{,} { \voltext}                    {volume}
    +{,} { }                            {publisher}
    +{,} { }                            {organization}
    +{,} { }                            {address}
    +{,} { }                            {status}
    +{,} { \PrintDOI}                   {doi}
    +{,} { \PrintISBNs}                 {isbn}
    +{}  { \parenthesize}               {language}
    +{}  { \PrintTranslation}           {translation}
    +{;} { \PrintReprint}               {reprint}
    +{,} { \PrintDate}                  {date}
    +{.} { }                            {note}
    +{.} {}                             {transition}
}

\BibSpec{article}{%
    +{}  {\PrintAuthors}                {author}
    +{,} { \textit}                     {title}
    +{.} { }                            {part}
    +{:} { \textit}                     {subtitle}
    +{,} { \PrintContributions}         {contribution}
    +{.} { \PrintPartials}              {partial}
    +{,} { }                            {journal}
    +{}  { \textbf}                     {volume}
    +{}  { \PrintDatePV}                {date}
    +{,} { \issuetext}                  {number}
    +{,} { \eprintpages}                {pages}
    +{,} { }                            {status}
    +{,} { \PrintDOI}                   {doi}
    +{}  { \parenthesize}               {language}
    +{}  { \PrintTranslation}           {translation}
    +{;} { \PrintReprint}               {reprint}
    +{.} { }                            {note}
    +{,} { \eprint}                     {eprint}
    +{.} {}                             {transition}
}
\BibSpec{collection.article}{%
    +{}  {\PrintAuthors}                {author}
    +{,} { \textit}                     {title}
    +{.} { }                            {part}
    +{:} { \textit}                     {subtitle}
    +{,} { \PrintContributions}         {contribution}
    +{,} { \PrintConference}            {conference}
    +{}  {\PrintBook}                   {book}
    +{,} { }                            {booktitle}
    +{,} { \PrintDateB}                 {date}
    +{,} { pp.~}                        {pages}
    +{,} { }                            {publisher}
    +{,} { }                            {organization}
    +{,} { }                            {address}
    +{,} { }                            {status}
    +{,} { \PrintDOI}                   {doi}
    +{,} { \eprint}        {eprint}
    +{}  { \parenthesize}               {language}
    +{}  { \PrintTranslation}           {translation}
    +{;} { \PrintReprint}               {reprint}
    +{.} { }                            {note}
    +{.} {}                             {transition}
}

\BibSpec{misc}{
  +{}{\PrintAuthors}  {author}
  +{,}{ \textit}      {title}
  +{,}{ }             {note}
  +{,}{ }             {date}
  +{.}{}              {transition}
}

\usepackage{accents}
\usepackage{makecell}
\usepackage{hhline}
\usepackage[mathscr]{eucal}
\newcommand\mapsfrom{\mathrel{\reflectbox{\ensuremath{\mapsto}}}}

\DeclareMathOperator{\gBr}{\hspace{-0.4ex}\widehat{Br}}

\DeclareMathOperator{\gTw}{\widehat{Tw}}
\newcommand{\lu}[1]{{}^{#1} \hspace{-0.1ex} }
\renewcommand{\blank}{\text{\textvisiblespace}}
\newcommand{\bdot}[1]{\accentset{\mbox{\large\bfseries .}}{#1}}

\DeclareMathOperator{\hotimes}{\hat{\otimes}}
\newcommand{\F}{\mathscr{F}}

\newcommand{\K}{\mathrm{K}}
\newcommand{\bK}{\mathbf{K}}
\newcommand{\KR}{\mathrm{KR}}
\newcommand{\KF}{\mathrm{KF}}

\newcommand{\KK}{\mathrm{KK}}
\newcommand{\Salg}{\mathcal{S}}
\newcommand{\KKR}{\mathrm{KKR}}
\newcommand{\KKQ}{\mathrm{KKQ}}
\newcommand{\G}{\mathcal{G}}
\newcommand{\T}{\mathbb{T}}

\newcommand{\Bop}{\mathbb{B}}
\newcommand{\Lop}{\mathbb{L}}
\newcommand{\Kop}{\mathbb{K}}
\newcommand{\Mop}{\mathbb{M}}
\newcommand{\Cst}{\mathrm{C}^*}

\newcommand{\Hilb}{\mathscr{H}}
\newcommand{\U}{\mathcal{U}}

\newcommand{\ev}{\mathrm{ev}}
\newcommand{\vD}{\mathrm{vD}}
\newcommand{\Kar}{\mathrm{Kar}}
\newcommand{\Zt}{\mathbb{Z}_2}
\newcommand{\Cl}{\mathrm{C}\ell}

\let\Im\relax
\DeclareMathOperator{\Im}{\mathrm{Im}}
\DeclareMathOperator{\pr}{pr}
\DeclareMathOperator{\qAut}{Aut_{qtm}}
\DeclareMathOperator{\Ad}{Ad}

\newcommand{\sA}{\mathscr{A}}
\newcommand{\sC}{\mathscr{C}}
\newcommand{\sG}{\mathscr{G}}
\newcommand{\sT}{\mathscr{T}}
\newcommand{\sP}{\mathscr{P}}
\title[Notes on twisted equivariant $\K$-theory for $\Cst$-algebras]{Notes on twisted equivariant $\K$-theory for $\Cst$-algebras}
\author[Y. Kubota]{Yosuke Kubota}
\address{Graduate School of Mathematical Science, The University of Tokyo, 3-8-1 Komaba, Meguro-ku, Tokyo 153-8914, Japan}
\email{ykubota@ms.u-tokyo.ac.jp}
\allowdisplaybreaks[2]
\date{9 February, 2016}

\begin{document}
\maketitle
\begin{abstract}
In this paper, we study a generalization of twisted (groupoid) equivariant $\K$-theory in the sense of Freed-Moore for $\Zt$-graded $\Cst$-algebras. It is defined by using Fredholm operators on Hilbert modules with twisted representations. We compare it with another description using odd symmetries, which is a generalization of van Daele's $\K$-theory for $\Zt$-graded Banach algebras. In particular, we obtain a simple presentation of the twisted equivariant $\K$-group when the $\Cst$-algebra is trivially graded. It is applied for the bulk-edge correspondence of topological insulators with CT-type symmetries.
\end{abstract}
\tableofcontents
\section{Introduction}

Recently, there has been an increasing interests in relations between $\K$-theory and the theory of topological insulators, an area of solid state physics. According to Kitaev's periodic table \cite{K2009}, free fermion phases are topologically classified in $10$ types, each of which corresponds to one of $2$ complex $\K$-groups and $8$ real $\K$-groups. This classification is formulated by Freed-Moore~\cite{MR3119923} in terms of twisted $\K$-theory. A key idea is Wigner's theorem~\cite{MR0106711} (see also Section 1 of \cite{MR3119923}), which asserts that a quantum symmetry is given by a linear/antilinear and even/odd projective representation of a groupoid. Freed and Moore introduce a generalized version of twisted equivariant $\K$-theory of groupoids classifying these representations. 

Here, twisted $\K$-theory \citelist{\cite{MR0282363}\cite{MR1018964}\cite{MR2172633}\cite{MR2119241}\cite{MR2513335}\cite{MR2860342}\cite{MR3119923}} is a kind of cohomology theory determined by spaces (or groupoids in general) and twists on it, that is, triplets $(\phi,c,\tau)$ where $\phi$ and $c$ are homomorphisms from $\G$ to $\Zt$ and $\tau$ is a $\phi$-twisted central extension of $\G$ by $\mathbb{T}$. Roughly speaking, twisted $\K$-theory associated to these data classifies $\Zt$-graded $\tau$-projective representations of $\G$ whose $\Zt$-grading and linearity are determined by $c$ and $\phi$. For a fixed twist $(\phi,c,\tau)$ of $\G$, the twisted $\K$-group of the action groupoid $\G \ltimes X$ with respect to the pull-back of $(\phi,c,\tau)$ is a topological invariant of a $\G$-space $X$. It enables us to regard twisted $\K$-theory as a functor from the category of locally compact Hausdorff $\G$-spaces to the category of $R(\G)$-modules.

In this paper, we investigate a canonical generalization of groupoid equivariant $\K$-theory of $\Cst$-algebras for twisted equivariant setting in the sense of \cite{MR3119923}. After reviewing a classification of twists in this sense by the \v{C}ech cohomology groups in Section \ref{section:2}, we start with the generalization of Kasparov's $\KK$-theory~\cite{MR582160} in Section \ref{section:3}. A general framework for twisted equivariant $\KK$-theory is introduced by Chabert-Echterhoff~\cite{MR1857079} as a bifunctor from the category of twisted $G$-$\Cst$-algebras to the category of $R(G)$-modules. In contrast, we define the $\lu \phi \KK ^\G _{c,\tau}$-group as a bifunctor from the category of ($\phi$-twisted) $\G$-$\Cst$-algebras. The group $\lu \phi \KK ^\G _{c,\tau}(A,B)$ is evidently suitable as a generalization of the $\KK$-group because it is actually isomorphic to a certain $\KKR$-group in the sense of Moutuou~\cite{MR3177819} (Proposition \ref{prp:detwist}).

In Section \ref{section:4}, we define the twisted equivariant $\K$-group as the twisted equivariant $\KK$-group $\lu \phi \KK ^\G _{c,\tau}(\real ,A)$. For example, when $A$ is the continuous function algebra of a compact Hausdorff $\G$-space $X$, it is isomorphic to the set of homotopy classes of twisted $\G$-equivariant families of Fredholm operators on $X$, which is a generalization of Atiyah's formulation of $\K$-theory using Fredholm operators \citelist{\cite{MR1043170}\cite{MR2172633}}. Moreover, we obtain a generalization of the Green-Julg theorem (Theorem \ref{thm:GJ}).  

It is natural to expect that this new $\K$-theory is presented by using ``finite dimensional'' objects such as vector bundles or projections instead of Fredholm operators. In general, it does not go on even if the groupoid is proper. However, when the groupoid and its central extension has enough finite dimensional representations, a desired presentation is given in Section \ref{section:5} as a generalization of van Daele's formulation of $\K$-theory for $\Zt$-graded Banach algebras. Here, the boundary map of the long exact sequence is also presented in a simple way by the exponential map. Moreover, we obtain a more simple presentation of twisted equivariant $\K$-groups for ungraded $\Cst$-algebras which is similar to Karoubi's $\K$-theory~\citelist{\cite{MR0238927}\cite{MR2458205}}. 

This presentation of the twisted equivariant $\K$-groups is applied for a topological classification of gapped Hamiltonians of fermionic quantum systems. The boundary map of the Toeplitz exact sequence gives the bulk-edge correspondence of topological insulators. In \cite{Kubota2}, the author consider the coarse Mayer-Vietoris exact sequence for twisted equivariant $\K$-groups of Roe algebras in order to prove the bulk-edge correspondence for quantum systems which is not translation-invariant such as quasi-crystals. 

In this context, Kellendonk~\cite{mathKT150906271} gives a detailed calculation of van Daele's $\K$-groups for certain $\Zt$-graded Real $\Cst$-algebras in connection with the classification of topological phases for each of $10$ symmetry types in Kitaev's periodic table. 
We remark that these groups are the same thing as the description given in Theorem \ref{cor:triv} of the twisted equivariant $\K$-groups when we consider the group $\sA$ and a twist $(\phi,c,\tau)$ as in Example \ref{exmp:CT}.

\subsection*{Notations}
We use the following notations throughout this paper.
\begin{itemize}
\item We assume that (topological) groupoids $\G=(\G^0,\G^1,s,r)$ are second countable, locally compact and Hausdorff groupoids with a Haar system (Definition 2.2 of \cite{MR584266}). A groupoid is proper if the map $(s,r) : \G^1 \to \G ^0 \times \G^0$ is proper. We use the convention $\G _x^y:=(s,r)^{-1}(x,y)$, $\G ^2:=\G ^1 {}_s \hspace{-0.2em}\times_{r} \G^1$ i.e.\ $g \circ h$ is well-defined if $s(g)=r(h)$.  
\item For a pre-simplicial space $M_\bullet$ (see Section 2 of \cite{MR2231869}), we write $\tilde{\varepsilon}_i:M_n \to M_{n-1}$ ($i=0,\dots, n$) for the face maps induced from the unique increasing map $\varepsilon _i$ that avoids $i$.
\item For a $\Cst$-algebra $A$, let $A_\real$ denote the underlying real $\Cst$-algebra. We use the same symbol $A_\real$ for the Real $\Cst$-algebra $A_\real \otimes _\real \comp$.
\item For a Hilbert space $\Hilb$ (resp.\ a Hilbert $\Cst$-module), let $\Kop (\Hilb)$ and $\Bop (\Hilb)$ (resp.\ $\Lop (\Hilb)$) denote the compact operator algebra and the bounded operator algebra on $\Hilb$. We write $\hat \Mop _{p,q}$ for the (Real) matrix algebra $\Kop (\real ^n \oplus (\real ^{\mathrm{op}})^m)$. \item We write as $\alpha _g^A$ (or simply $\alpha _g$ unless it causes some confusion) for a group action on a $\Cst$-algebra $A$. Similarly, we write as $u _g^E$ (or simply $u _g$) for a representation on a Hilbert $\Cst$-module $E$.
\item Throughout this paper we deal with $\Zt$-graded $\Cst$-algebras, that is, $\Cst$-algebras together with the involutive $\ast$-automorphism $\gamma _A$ (or simply $\gamma$). Let $A ^{\mathrm{even}}$ and $A ^{\mathrm{odd}}$ denote the $+1$ and $-1$ eigenspaces of $\gamma$ respectively. We use the notation $[a,b]$ for the supercommutator $[a,b]:=ab - (-1)^{|a| \cdot |b|}ba$.
\item For a complex number $\lambda$ (resp.\ a complex line bundle, an element in a Real $\Cst$-algebra), we write as $\lu{0}\lambda:=\lambda$ and $\lu{1}\lambda :=\overline{\lambda}$. In the same way, we say a map between vector spaces is $0$-linear (resp.\ $1$-linear) if it is linear (resp.\ antilinear).
\item Let $\Cl _{n,m}$ denote the Real Clifford algebra associated to $\real ^{n+m}$ together with the inner product $\ebk{\xi,\xi}=-(\xi _1)^2-\dots -(\xi _n)^2+(\xi _{n+1})^2+\dots +(\xi _{n+m})^2$. That is, $\Cl _{n,m}$ is the finite dimensional Real $\Cst$-algebra generated by odd Real elements $f_1,\dots ,f_n$ and $e_1,\dots ,e_m$ with relations $[e_i,e_j]=2\delta _{i,j}$, $[f_i,f_j]=-2\delta _{i,j}$ and $[e_i,f_j]=0$. Let $C_{n,m}$ denote the Clifford group, that is, the subgroup of $\Cl _{n,m}^\times $ generated by $\{ e_1,\dots ,e_m, f_1 ,\dots ,f_n\}$. 
\end{itemize}

\section{Extensions and Brauer groups on groupoids}\label{section:2}
In this section, we review relations between twists, Brauer groups and \v{C}ech cohomology groups and generalize them for twists in the sense of Freed-Moore~\cite{MR3119923}. Most arguments are based on the works on \cite{MR1646047} and \cite{MR2231869}. 

First of all, we show the conclusion of this section. 
\begin{thm}\label{thm:twist}
Let $\G$ be a groupoid and let $\phi$ be a $\Zt$-valued $1$-cocycle of $\G$. Then, the following three groups are canonically isomorphic:
\begin{enumerate}
\item the $\phi$-twisted graded Brauer group $\lu{\phi}\gBr (\G)$,
\item the group $\lu \phi \gTw (\G)$ of twists on $\G$ associated to $\phi$,
\item the \v{C}ech cohomology group $\check{H}^1(\G ;\zahl _2) \oplus _\epsilon \check{H}^2(\G ;\T _\phi)$.
\end{enumerate}
\end{thm}
Consequently, twists of $\G$ are classified by
\[(\phi , c, \tau) \in \coprod _{\phi \in \check{H}^1(\G;\zahl _2)}\check{H}^1(\G ;\zahl _2) \oplus _\epsilon \check{H}^1(\G ; \T _\phi).\]

Before the proof, we introduce the definition of each object appearing in the statement of Theorem \ref{thm:twist}. 

The \v{C}ech cohomology groups for groupoids $\G$ and $\G$-sheaves are the same thing as the one introduced in the Section 4 of \cite{MR2231869}. Let $G$ be an abelian group. By Proposition 5.2 of \cite{MR2231869}, a $G$-valued $1$-cocycle $\phi$ corresponds to a principal $G$-bundle $p_\phi : \G ^0_\phi \to \G ^0$ with $\G$-action. In other words, $\phi$ determines a Hilsum-Skandalis map \cite{MR925720} from $\G$ to $G$. By replacing $\G$ with another groupoid $\G_\phi :=\G \ltimes \G_\phi ^0 \rtimes G$ which is Morita equivalent to $\G$, we may assume that $\phi$ is represented by a groupoid homomorphism $\phi : \G_\phi \to G$. 

Let $\epsilon$ denote the $2$-cocycle on the abelian group $\check{H}^1(\G ;\zahl _2)$ which takes value in $\check{H}^2(\G ;\T _\phi)$ determined by
\[ \epsilon (c_1,c_2)(g,h):=(-1)^{c_2(g)\cdot c_1(h)} \]
and let $\check{H}^1(\G ;\zahl _2) \oplus _\epsilon \check{H}^2(\G ;\T _\phi)$ denote the central extension group corresponding to $\epsilon$.

\begin{defn}[Definition 7.23 of \cite{MR3119923}]
Let $\G$ be a groupoid and let $\phi$ be a groupoid homomorphism from $\G$ to $\Zt$. 
\begin{itemize}
\item A $\Zt$-grading of $\G$ is a groupoid homomorphism $c : \G \to \Zt$. 
\item A $\phi$-twisted extension of $\G$ is a principal $\T$-bundle $\tau: (\G^\tau)^1 \to \G ^1$ together with the linear isomorphism
\[\lambda _{g_2,g_1}:\lu{\phi(g_1)}L_{g_2}^\tau \otimes L_{g_1}^\tau \to L_{g_2g_1}^\tau\]
(where $L^\tau :=(\G ^\tau)^1 \times _{\T} \comp$) such that the digram
\[
\xymatrix@C=5em{
\lu{\phi (g_2g_1)}L_{g_3}^\tau \otimes \lu{\phi (g_1)} L_{g_2}^\tau \otimes L_{g_1}^\tau \ar[r]^{ \ \ \ \ \id \otimes \lambda _{g_2,g_1}} \ar[d]^{\lambda _{g_3,g_2} \otimes \id}& L_{g_3}^\tau \otimes L_{g_2g_1}^\tau \ar[d]^{\lambda _{g_3,g_2g_1}} \\
\lu{\phi (g_1)} L_{g_3g_2}^\tau \otimes L_{g_1}^\tau \ar[r]^{\lambda{g_3,g_2g_1}} &L_{g_3g_2g_1}^\tau
}
\]
commutes. 
\end{itemize}
We say that a triple $(\phi,c,\tau)$ as above is a \emph{twist} of $\G$. 
\end{defn}

\begin{exmp}
Let $\tau$ be a central extension of $\G$ by $\Zt$. Then, the associated bundle $\G^\tau \times _{\Zt} \T$ with respect to the inclusion $\Zt \cong \{ \pm 1 \} \subset \T$ is a $\phi$-twisted central extension of $\G$ by the product $(g, t) \cdot (h,s) :=(gh, \lu{\phi (g)}st)$. 
\end{exmp}

\begin{exmp}\label{exmp:real}
Consider the group $\G =\Zt$ with the generator $\underline{b}$ and the group homomorphism $\phi _\real :=\id : \Zt \to \Zt$. Then, the isomorphism class of a $\phi _\real$-twisted central extension $\tau$ of $\Zt$ by $\mathbb{T}$ is determined by the square $b^2 \in \T$ of a lift $b$ of $\underline{b}$. By the relation $b \cdot b^2=b^2 \cdot b= b \overline{b^2}$, the scalar $b ^2$ must be $\pm 1$. The corresponding two $\phi_\real$-twisted extensions are denoted by $\tau _\real ^0$ and $\tau _{\mathbb{H}}^0$ respectively.
For a central extension $\tau$ of $\G$ by $\Zt$, let $\tau _\real$ and $\tau _\mathbb{H}$ denote central extensions $\tau \times \tau _{\real}^0$ and $\tau \times \tau _\mathbb{H}^0$ of $\G _\real := \G \times \Zt$ respectively. 
\end{exmp}

\begin{exmp}\label{exmp:Cl}
Let $F_{n,m}$ denote the group $\Zt^{n+m}=\Zt \underline{f}_1 \oplus \dots \oplus \Zt \underline{f}_ n \oplus \Zt \underline{e}_1\oplus \dots \oplus \Zt \underline{e}_m $. Then, the canonical homomorphism from the Clifford group $C_{n,m}$ to $F_{n,m}$ mapping $e_i$ and $f_j$ to $\underline{e}_i$ and $\underline{f}_j$ respectively determines a central extension of $F_{n,m}$ by $\Zt$. We write $\tau _{n,m}^0$ for this extension and $c_{n,m}^0$ for the group homomorphism $F_{n,m} \to \Zt$ given by $c_{n,m}(\underline{e}_i)=c_{n,m}(\underline{f}_j)=1$. 
\end{exmp}

\begin{exmp}\label{exmp:CT}
Let $\sG$ denote the finite abelian group $\Zt \times \Zt$ with the generator $\underline{C}:=(1,0)$ and $\underline{T}:=(0,1)$. There are three proper subgroups $\sC$, $\sP$ and $\sT$ of $\sG$ generated by $\underline{C}$, $\underline{C} \underline{T}$ and $\underline{T}$. We fix two group homomorphisms $\phi_\sG , c_\sG :\sG \to \Zt$ defined by $\phi _\sG (\underline{T})=\phi _\sG (\underline{C})=1$, $c _\sG(\underline{C})=1$ and $c _\sG(\underline{T})=0$. 
According to Lemma 6.17 and Proposition 6.4 of \cite{MR3119923}, every central extension $\mathscr{A}^\tau$ of a subgroup $\sA$ of $\sG$ by $\Zt$ has a unique lift $C,T$ of $\underline{C}, \underline{T}$ such that $(CT)^2=1$ and the pairs $(\sA, \tau)$ are classified by the choice of signs $C^2=\pm 1$ and $T^2=\pm 1$. 
We say that a \emph{CT-type symmetry} is a quadruple $(\G,\phi,c,\tau)$ where $\G =\G_0 \times \mathscr{A}$, $(\phi,c)=(\phi _\sG, c_\sG ) \circ \pr _2$ and $ \G ^\tau =\G _0^{\tau _0}\times \sA^{\tau_\mathscr{A}}$. For a CT-type symmetry, the pair $(\mathscr{A} , \tau _{\mathscr{A}})$ is called its \emph{CT-type}. 
\end{exmp}

For a groupoid homomorphism $\phi : \G \to \Zt$, let $\T _\phi$ denote the $\G$-module $\G \times \T$ with the $\G$-action $\alpha _g(t)=\lu {\phi (g)}t$. Let $\mathrm{ext}(\G,\T_\phi)$ denote the set of equivalent classes of $\phi$-twisted extensions, which is an abelian group under the Baer sum $[L_1] + [L_2]:=[L_1 \otimes L_2]$. Let $\lu \phi \widehat{\mathrm{tw}} (\G)$ denote the group of twists $\Hom (\G , \Zt) \oplus _\varepsilon \mathrm{ext}(\G,\T_\phi)$ associated to $\phi$. Here, for $c_1,c_2 \in \Hom (\G , \Zt)$, $\varepsilon (c_1,c_2)$ denotes the central extension $\G \times \Zt$ of $\G$ by $\Zt$ with the product $(g,i)\cdot (h,j):=(gh,i+j+c_2(g)c_1(h))$.

The group $\lu \phi \widehat{\mathrm{tw}} (\G)$ is related to \v{C}ech cohomology groups by the following way. For a groupoid $\G$ and an open covering $\U $ of $\G ^0$, $\G[\U]$ denotes the groupoid given by $\G[\U]^0:=\coprod _{U \in \U} U$, 
\[
\G[\U]^1:=\{ (y,g,x) \in \G[\U]^0 \times \G^1  \times \G[\U]^0 \mid s(g)=\pi(x), r(g)=\pi (y) \}
\]
where $\pi$ is the canonical projection from $\G[\U]^0$ to $\G$, $s$ and $r$ are the third and first projections and $(z,g,y)\cdot (y,h,x):=(z,gh,x)$. Remark that this groupoid is Morita equivalent to $\G$. 

\begin{prp}[Proposition 5.2 and Proposition 5.6 of \cite{MR2231869}]
The inductive limit
\[
\lu \phi \gTw (\G) := \varinjlim _{\U}\lu \phi \widehat{\mathrm{tw}}(\G[\U])
\]
is isomorphic to $\check{H}^1(\G;\Zt) \oplus _\epsilon \check{H}^2(\G; \T_\phi)$. 
\end{prp}

Next, we turn to the classification of certain $\phi$-twisted $\G$-equivariant bundles of $\Cst$-algebras. Before that, we prepare terminologies of linear actions and representations of groupoids for the convenience of readers. Basic references are Section 13--14 of \cite{MR936628} and \cite{MR1646047}.

Let $X$ be a locally compact second countable Hausdorff space. A \emph{Banach bundle} or a \emph{continuous field of Banach spaces} over $X$ is a topological space $\mathcal{E}$ together with a continuous open surjection $p:\mathcal{E} \to X$, continuous maps $\ssbk{\blank}:\mathcal{E} \to \real _+$, $\blank +\blank : \mathcal{E} \times _X \mathcal{E} \to \mathcal{E}$ and $\blank \cdot \blank : \comp \times \mathcal{E} \to \mathcal{E}$ such that these operations determine the Banach space structure on each fiber $\mathcal{E}_x:=p^{-1}(x)$ and any nets $\xi_j$ in $\mathcal{E}$ such that $p(\xi_j) \to x$ and $\ssbk{\xi_j} \to 0$ converges to $0_x$. For a Banach bundle $\mathcal{E}$, the space $\Gamma_0 (X,\mathcal{E})$ of continuous sections vanishing at infinity is a Banach space which is embedded in $\prod _{x \in X} \mathcal{E}_x$. It is known that a Banach bundle $\mathcal{E}$ over a locally compact Hausdorff space $X$ has enough continuous cross-sections, that is, for any $\xi_x \in E_x$ there is a continuous section $\xi \in \Gamma (X,E)$ such that $\xi (x)=\xi _x$ (see Appendix C of \cite{MR936628}).  

A (strongly continuous) \emph{$\phi$-twisted $\G$-Banach bundle} is a Banach bundle $\mathcal{E}$ over $\G ^0$ together with a continuous bundle map $\alpha : s^*\mathcal{E} \to r^*\mathcal{E}$ such that each $\alpha _g : \mathcal{E}_{s(g)} \to \mathcal{E}_{r(g)}$ is a bounded linear map, the map $g \mapsto \ssbk{\alpha _g(\xi (s(g)))}$ is continuous for any $\xi \in \Gamma _0(X, \mathcal{E})$ and $\alpha _{g}\alpha _{h}=\alpha _{gh}$, $\alpha _{\id _x}=\id _{\mathcal{E}_x}$. We say that a $\real$-linear action $\alpha$ of $\G$ on a $\Zt$-graded Banach bundle $\mathcal{E}$ is \emph{$\phi$-linear} if each $\alpha _g$ is $\phi (g)$-linear and \emph{$c$-graded} if each $\alpha _g$ is even if $c(g)=0$ and odd $c(g)=1$. 

In this paper, we mainly deal with $\G$-bundles of Hilbert spaces and $\Cst$-algebras. 
A \emph{Hilbert bundle} $\mathcal{H}$ over $X$ is a Banach bundle together with the inner product $\ebk{\blank , \blank}: \mathcal{H} \times _X \mathcal{H} \to \comp$ such that $\ssbk{\xi}^2=\ebk{\xi,\xi}$ for any $\xi \in \mathcal{H}$. It is shown in Proposition 2 of \cite{MR2917284} that the section space $\Gamma _0(X, \mathcal{H})$ has a canonical Hilbert $C_0(X)$-module structure and every Hilbert $C_0(X)$-module is isomorphic to the section space of a Hilbert bundle. We say that a linear action $u$ of $\G$ on a Hilbert bundle $\mathcal{H}$ over $\G^0$ is unitary if each $u_g$ is a $\phi$-linear unitary operator (i.e.\ $u_g$ is bijective and satisfies $(u_g\xi,u_g\eta)=(\xi,\eta)$ for any $\xi , \eta \in \mathcal{H}_{s(g)}$).

\begin{defn}
A \emph{$(\phi,c,\tau)$-twisted $\G$-Hilbert bundle} is a Hilbert bundle $\mathcal{H}$ over $\G ^0$ together with a $\phi$-linear $c$-graded unitary action $u$ of $\G ^\tau$ on $\mathcal{H}$ whose restriction on ${\mathbb{T}}$ is given by the complex multiplication. We say that a $(\phi,c,\tau)$-twisted $\G$-Hilbert bundle is a \emph{$(\phi,c,\tau)$-twisted unitary representation} of $\G$ if the underlying Hilbert bundle is locally trivial whose structure group is the unitary group $\mathcal{U}(\Hilb _0)$ with the strong topology.
\end{defn}
Here, we use the letter $\Hilb _0$ for the abstract separable infinite dimensional Hilbert space $\ell ^2 \nat$. By the Kasparov stabilization theorem, for any non-trivial Hilbert bundle $\mathcal{H}$ over a locally compact second countable space $X$, $\mathcal{H} \otimes \Hilb _0$ is isomorphic to the trivial bundle $X \times \Hilb _0$ (cf.\ Proposition 7.4 of \cite{MR1325694}). Therefore, for any $(\phi,c,\tau)$-twisted Hilbert $\G$-bundle $\mathcal{H}$, $\mathcal{H} \otimes \Hilb _0$ is a $(\phi,c,\tau)$-twisted unitary representation.

\begin{exmp}
Set
\[C_c(\G, L^\tau):=\{ \xi \in C_c(\G ^\tau) \mid \xi (t^{-1} g)=\lu {\phi (g)} t \cdot  \xi (g) \text{ for any $t \in \T$}\}, \]
which is a $\Zt$-graded complex vector space by the scalar multiplication $(\lambda \cdot \xi)(g):=\lu{\phi (g)}\lambda \xi (g)$ and the $\Zt$-grading $(\gamma \xi )(g):=(-1)^{c(g)}\xi (g)$.
Let $\lu \phi L^2_{c,\tau}\G$ be the Hilbert $C_0(\G ^0)$-module given by the completion of $C_c(\G)$ with respect to the inner product $\ebk{\xi,\eta}(x)=\int \xi (g)\overline{\eta (g)} d\lambda _x (g)$.
Then, the corresponding Hilbert bundle over $\G ^0$ is a $(\phi ,c,\tau )$-twisted $\G$-Hilbert bundle whose fiber on $x$ is $L^2 (\G _x , L^\tau) $ with the left regular representation $u_g(\xi)(h)=\xi(g^{-1} h)$. 
Set $\lu \phi \Hilb _{\G}^{c,\tau}:=\lu \phi L^2_{c,\tau}\G \hotimes \hat{\Hilb} _0$ and let $\lu \phi \mathcal{H}_\G ^{c,\tau}$ denote the corresponding $(\phi,c,\tau)$-twisted unitary representation. 
\end{exmp}

A \emph{$\phi$-twisted $\G$-$\Cst$-bundle} is a $\phi$-twisted $\G$-Banach bundle $(\mathcal{A},\alpha)$ together with continuous maps $\blank \cdot \blank : \mathcal{A} \times _X \mathcal{A} \to \mathcal{A}$ and $\ast : \mathcal{A} \to \mathcal{A}$ such that each $\mathcal{A}_x$ is a $\Cst$-algebra and each $\alpha _g: \mathcal{A}_{s(g)} \to \mathcal{A}_{r(g)}$ is a $\ast$-homomorphism by these operations. 
The space of continuous sections vanishing at infinity $\Gamma _0(X,\mathcal{A})$ of a $\phi$-twisted $\G$-$\Cst$-bundle $\mathcal{A}$ is a $\phi$-twisted $\G$-$\Cst$-algebra (which will be introduced in Definition \ref{def:GCst}) by the fiberwise $\ast$-operation, the sup norm and the canonical $\G$-action. 

We say that two $\phi$-twisted $\G$-$\Cst$-algebras $A$ and $B$ are Morita-equivalent if there is a $\phi$-twisted $\Zt$-graded $\G$-equivariant Hilbert $B$-module $E$ (Definition \ref{def:Gmod}) such that $\ebk{E,E}=B$ and $\Kop (E)\cong A$. Two $\phi$-twisted $\G$-$\Cst$-bundles are Morita-equivalent if the corresponding $\G$-$\Cst$-algebras are Morita-equivalent.

\begin{defn}\label{def:bdl}
We say that a $\G$-$\Cst$-bundle $\mathcal{A}$ is ($\Zt$-graded) \emph{elementary} if each fiber is isomorphic to the algebra of compact operators on $\Zt$-graded separable Hilbert spaces and satisfies \emph{Fell's condition}, that is, for any $x \in \G ^0$ there is a neighborhood $U$ of $x$ and a section of rank $1$ projections on $U$. The \emph{$\phi$-twisted Brauer group} $\lu{\phi}\gBr (\G)$ of $\G$ is the group of Morita-equivalent classes of elementary $\phi$-twisted $\G$-$\Cst$-bundles under the product given by the graded tensor product.
\end{defn}
Note that the underlying $\Cst$-bundle over $\G ^0$ of a stable (i.e.\ $\mathcal{A} \cong \mathcal{A} \otimes \Kop$) elementary $\G$-$\Cst$-bundle $\mathcal{A}$ is locally trivial (see Theorem 10.7.15 of \cite{MR0458185}).

\begin{exmp}\label{exmp:rep}
Let $\mathcal{H}$ be a $(\phi,c,\tau)$-twisted $\G$-Hilbert bundle and $\Hilb :=\Gamma _0(\G ^0,\mathcal{H})$. Then, $\Kop (\Hilb)$ is isomorphic to the section algebra of a $\phi$-twisted $\G$-$\Cst$-bundle (see 10.7.1 of \cite{MR0458185}). Hereafter we use the same symbol $\Kop (\Hilb)$ for this $\phi$-twisted $\G$-$\Cst$-bundle. Moreover, it satisfies Fell's condition because a local section of rank $1$ projections in $\Kop (\Hilb)$ corresponds to a continuous non-zero local section of $\mathcal{H}$. For two $(\phi,c,\tau)$-twisted $\G$-$\Cst$-bundles $\Hilb _1$ and $\Hilb _2$, $\Kop (\Hilb _1)$ and $\Kop (\Hilb _2)$ are Morita-equivalent since $\Kop (\Hilb _1, \Hilb _2)$ is a $\Kop (\Hilb _1)$-$\Kop (\Hilb _2)$ imprimitivity bimodule. Therefore, we have a map $\lu \phi \gTw (\G) \to \lu \phi \gBr (\G)$ given by $(c,\tau) \mapsto  \lu \phi \Kop _{\G}^{c,\tau} :=\Kop (\lu \phi \Hilb _\G ^{c,\tau})$. This map is actually a group homomorphism since $\Kop (\Hilb _1) \hotimes \Kop (\Hilb _2) \cong \Kop (\Hilb _1 \hotimes \Hilb _2)$ and $\Hilb _1 \hotimes \Hilb _2$ is a $(\phi,c,\tau)$-twisted Hilbert $\G$-bundle where $c=c_1+c_2$ and $\tau=\tau _1 +\tau _2 +\varepsilon (c_1,c_2)$ as is shown in Lemma \ref{lem:tensor} in Section \ref{section:3}.
\end{exmp}

\begin{proof}[Proof of Theorem \ref{thm:twist}]
The rest part of the proof is constructing the inverse $\lu \phi \gBr (\G) \to \lu \phi \gTw (\G)$ of the homomorphism in Example \ref{exmp:rep}, which is given in the same way as in Theorem 10.1 of \cite{MR1646047}. We only remark that for a $\phi$-twisted elementary $\G$-$\Cst$-bundle whose underlying $\Cst$-bundle is isomorphic to $C_0(\G ^0) \hotimes \Kop (\Hilb _0)$ (and hence the $\G$-action is given by a groupoid homomorphism $\pi : \G ^1 \to \qAut \Kop (\Hilb _0)$), the groupoid 
\[
E(\pi):=\{ (g,u) \in \G \times \qAut \Hilb_0 \mid \pi (g)=\Ad (u) \}
\]
is a $\phi$-twisted central extension of $\G$ by $\mathbb{T}$ and the canonical action of $E(\pi)$ on $C_0(\G ^0) \hotimes \Hilb _0$ determines a $(\phi,c,\tau)$-twisted $\G$-Hilbert bundle. Here $\qAut \Kop (\Hilb _0)$ denotes the group of linear/antilinear even $\ast$-automorphisms on $\Kop (\Hilb _0)$ and $\qAut \Hilb _0$ denotes the group of linear/antilinear and even/odd unitary operators on $\Hilb _0$. 
\end{proof}

For a $\phi$-twisted elementary $\G$-$\Cst$-bundle $\mathcal{A}$, we call the corresponding element in $\check{H} ^1(\G ; \Zt ) \oplus _\epsilon  \check{H}^2 (\G ; \T _\phi) $ is the \emph{Dixmier-Douady invariant} of $\mathcal{A}$. 

\section{Twisted equivariant $\KK$-theory}\label{section:3}
In this section, we introduce the notion of twisted equivariant $\KK$-theory for pairs of $\phi$-twisted $\G$-$\Cst$-algebras.  
We start with the definition of $\phi$-twisted actions of groupoids on $\Cst$-algebras. 
For a locally compact Hausdorff space $X$, a $\Zt$-graded $X$-$\Cst$-algebra is a $\Zt$-graded $\Cst$-algebra $A$ together with a unital $\Zt$-graded $\ast$-homomorphism $C_b(X) \to ZM(A)$. For example, the section algebra $\Gamma _0(X, \mathcal{A}) $ of a $\Cst$-bundle over $X$ is a $X$-$\Cst$-algebra. Conversely, by Proposition A.3 of \cite{MR2119241}, a $X$-$\Cst$-algebra $A$ such that $C_0(X)A=A$ is isomorphic to the section algebra of an upper-semicontinuous $X$-$\Cst$-bundle (a variation of $X$-$\Cst$-bundle such that the norm $\ssbk{\blank}: \mathcal{A} \to \real _+$ is upper-semicontinuous). Here, the fiber $A_x$ of the corresponding upper-semicontinuous $\Cst$-bundle is given by
\[
A_x:=A/(\{ f \in C_b(X) \mid f(x)=0 \} A).
\]
Moreover, bundle maps $\mathcal{A} \to \mathcal{B}$ and pull-backs $f^*\mathcal{A}$ by a continuous map $f: Y \to X$ of $\Cst$-bundles correspond to $X$-$\ast$-homomorphisms ($\ast$-homomorphisms $\varphi: A \to B$ satisfying $\varphi (fa)=f\varphi (a)$) and the tensor product $C(Y) \otimes _X A$ (Definition 1.5 and Definition 1.6 of \cite{MR918241}) respectively. 

\begin{defn}[D\'{e}finition 3.4 of \cite{MR1686846}]\label{def:GCst}
A \emph{$\phi$-twisted $\Zt$-graded $\G$-$\Cst$-algebra} is a $\Zt$-graded $\G^0$-$\Cst$-algebra $A$ together with a $\G^1$-$\ast $-homomorphism $\alpha : s^*A \to r^*A$ such that each $\alpha _g : A_{s(g)} \to A_{r(g)}$ is a $\phi(g)$-linear $\Zt$-graded $\ast$-homomorphism satisfying $\alpha _g \circ \alpha _h=\alpha _{gh}$ and $\alpha _{\id _x}=\id _{A_x}$. 
\end{defn}

\begin{remk}\label{remk:cpx}
Let $A$ be a $\phi$-twisted $\G$-$\Cst$-algebra. Then, by the same action $\alpha $, the underlying Real $\Cst$-algebra $A_\real $ is regarded as a Real $\G$-$\Cst$-algebra. Actually, the correspondence $a \oplus jb \mapsto (a +ib) \oplus (a-ib)$ gives an isomorphism between $A_\real \cong A \oplus jA$ (where $j$ is the imaginary unit of the Real $\Cst$-algebra $A_\real$) and $p_\phi ^*A$ as $\G$-$\Cst$-algebras. Moreover, the right $\Zt$-action on $p_\phi^* A$ corresponds to the complex conjugation on $A_\real$.
\end{remk}

\begin{exmp}
A Real $\G$-$\Cst$-algebra $(A,\alpha ')$ is regarded as a $\phi$-twisted $\G$-$\Cst$-algebra by the action $\alpha _g(x):=\lu{\phi (g)}\alpha _g'(x)$. In particular, for a Real $\Cst$-algebra $A$, the $\Cst$-algebra $C_b(\G ^0, A)$ together with the trivial $\G$-action is a ``trivial'' $\phi$-twisted $\G$-$\Cst$-algebra. For simplicity of notation, we use the same letter $A$ for this $\phi$-twisted $\G$-$\Cst$-algebra. Conversely, a $\phi _\real$-twisted $\G_\real$-$\Cst$-algebra is the same thing as a Real $\G$-$\Cst$-algebra.
\end{exmp}

\begin{defn}\label{def:Gmod}
Let $A$ be a $\Zt$-graded $\phi$-twisted $G$-$\Cst$-algebra. A \emph{$(\phi,c, \tau)$-twisted $\G$-equivariant Hilbert $A$-module} is a $\Zt$-graded Hilbert $A$-module $E$ with the $(\phi,c,\tau)$-twisted representation of $\G$, that is, a $\phi$-linear action $u : s^*E \to r^*E$ of $\G ^\tau$ such that $u|_\mathbb{T}$ is given by the complex multiplication, $\ebk{u_g \xi , u_g \eta}=\alpha _{\tau(g)} (\ebk{\xi,\eta})$ and $\gamma _E(u_g \xi)=(-1)^{c(g)}u_g\gamma _E(\xi)$ for any $g \in (\G^\tau)^1$.
\end{defn}

For example, a $(\phi,c,\tau)$-twisted $\G$-equivariant Hilbert $\real$-module is the same thing as the section space of a $(\phi,c,\tau)$-twisted $\G$-Hilbert bundle. In general, for a $\Zt$-graded $\phi$-twisted Hilbert $\G$-module $E$ over $A$ and a $(\phi,c,\tau)$-twisted $\G$-Hilbert bundle $\Hilb$, the tensor product $E \hotimes \Hilb$ is a $(\phi,c,\tau)$-twisted $\G$-equivariant Hilbert $A$-module. 

In general, the twisted $\G$-equivariant Hilbert $A$-module structure on the interior tensor product of two Hilbert modules is given as the following lemma. Here, let us choose a pre-simplicial open cover $\mathcal{U}^\bullet =\{ U^\mu _\bullet \}_{\mu \in I ^\bullet }$ of $\G ^\bullet$ (see Definition 4.1 of \cite{MR2231869}) such that $\mathcal{U}^1$ trivializes the fiber bundle $(\G ^{\tau})^1 \to \G ^1$. Then, by taking local continuous sections, we obtain locally defined actions $u^\mu : s^*E|_{U_1^\mu} \to r^*E|_{U_1^\mu}$ of $\G$ (instead of $\G^\tau$) for each $\mu \in I^1$ and a $\mathbb{T}_\phi$-valued $2$-cocycle $\tau =\{ \tau ^\nu \}_{\nu \in I^2}$ such that $u_g^{\tilde{\varepsilon} _2 \nu }u_h^{\tilde{\varepsilon} _0 \nu}=\tau ^{\nu} (g,h)u_{gh}^{\tilde{\varepsilon} _1\nu }$ on $U^\nu_2$ for any $\nu \in I^2$. For simplicity of notations, we often omit the superscript $\mu$ and $\nu$ for these locally defined actions and $2$-cocycles.

\begin{lem}\label{lem:tensor}
Let $A_i$ be $\Zt$-graded $\phi$-twisted $\G$-$\Cst$-algebras, let $(E_i, u^i)$ be $(\phi , c_i, \tau _i)$-twisted $\G$-equivariant Hilbert $A_i$-modules for $i=1,2$ and let $\varphi : A_1 \to \Lop (E_2)$ be a $\ast$-homomorphism such that 
\maa{\alpha ^{\Lop (E_2)} _g  \varphi (a) =(-1)^{|a| \cdot c(g)} \varphi (\alpha _g^{A_1}(a)) \label{form:superhom}}
for any homogeneous $a \in (A_1)_{s(g)}$. Then, 
\[
u_g  (\xi \hotimes \eta) =(-1)^{c_2(g) \cdot |\xi|} u_g^1 \xi \hotimes u_g ^2 \eta
\]
determines a well-defined $(\phi,c, \tau)$-twisted Hilbert $G$-module structure on $E:= E_1 \hotimes _{A_1}E_2$ where $c:=c_1+c_2$ and $\tau :=\tau _1 +\tau _2 + \epsilon (c_1,c_2)$.
\end{lem}
\begin{proof}
First we observe that the above $u_g$ determines a well-defined linear map on $E \otimes _{A_1,\mathrm{alg}} F$. Actually, for any $g \in \G$, $\xi \in (E_1)_{s(g)}$ and homogeneous elements $a \in (A_1)_{s(g)}$ and $\eta \in (E_2)_{s(g)}$,
\ma{
(-1)^{c_2(g) \cdot |\xi a|}(u_g^1(\xi a)) \hotimes (u_g^2\eta)&=(-1)^{c_2(g) \cdot (|\xi|+|a|)}(u_g^1 \xi) \alpha _g (a) \hotimes (u_g^2 \eta) \\
(-1)^{c_2(g) \cdot |\xi|}(u_g^1 \xi) \hotimes (u_g ^2(a\eta) )&=(-1)^{c_2(g) \cdot |\xi|}(u_g^1 \xi ) \hotimes ( (-1)^{c_2(g) \cdot |a|} \alpha _g(a) (u_g^2 \eta))\\
&=(-1)^{c_2(g) \cdot (|\xi|+|a|)}(u_g^1 \xi) \hotimes \alpha _g( a)(u_g^2 \eta).
}
Now for any homogeneous $\xi _i \in (E_1)_{s(g)}$ and $\eta _i \in (E_2)_{s(g)}$,
\ma{
\ebk{u_g(\xi_1 \hotimes \eta_1), u_g (\xi _2 \hotimes \eta _2)}&=(-1)^{c_2(g) \cdot (|\xi _1| +|\xi _2|)}\ebk{u_g^2 \eta_1, \ebk{u_g^1 \xi_1, u_g ^1 \xi_2}\cdot (u_g^2 \eta_2)}\\
&=(-1)^{c_2(g)\cdot (|\xi _1|+|\xi _2|)}\ebk{u_g^2 \eta_1, \alpha _g (\ebk{\xi_1,\xi_2})\cdot (u_g^2 \eta_2) }\\
&=\ebk{u_g^2 \eta _1, u_g^2 (\ebk{\xi _1,\xi _2}\cdot \eta _2)}\\
&=\alpha_g  (\ebk{ \xi_1 \hotimes \eta_1,\xi _2 \hotimes \eta_2})}
implies that the $G$-action on $E_1 \hotimes _{A,\mathrm{alg}}E_2$ preserves the inner product and hence extends to $E_1 \hotimes _AE_2$. The projective representation $u$ is $(\phi,c,\tau)$-twisted because
\ma{u_g u_h (\xi \hotimes \eta) &=(-1)^{c_2(h)|\xi|+c_2(g)|u_h \xi |}u_g^1u_h^1\xi \hotimes u_g^2u_h^2\eta\\
&=(-1)^{c_1(g)c_2(h)}\tau _1 (g,h) \tau _2 (g,h)u_{gh}(\xi \hotimes \eta)}
for any $\xi \in (E_1)_{s(h)}$ and $\eta \in (E_2)_{s(h)}$.
\end{proof}

\begin{exmp}\label{exmp:Morita}
For a $(\phi,c,\tau)$-twisted unitary representation $(\mathscr{V}, \pi )$ of $\G$, we write $(\bdot{\mathscr{V}}, \bdot \pi )$ for the $(\phi,c,\tau + \epsilon (c,c))$-twisted unitary representation given by $\bdot{\mathscr{V}}=\mathscr{V}$ as Hilbert spaces and $\bdot \pi (g):= \gamma ^{c(g)} \circ \pi (g)$. Then, $\varphi:=\id : \Kop (\mathscr{V}) \to \Kop (\bdot{\mathscr{V}})$ satisfies (\ref{form:superhom}). In this case, we have $\mathscr{V}^* \hotimes _{\Kop (\mathscr{V})} \bdot{\mathscr{V}} \cong \real$ as $\phi$-twisted $\real$-$\real$ bimodules and $\bdot{\mathscr{V}} \hotimes _\real \mathscr{V}^* \cong \Kop (\mathscr{V})$ as $\phi$-twisted $\Kop (\mathscr{V})$-$\Kop (\mathscr{V})$ bimodules. 
\end{exmp}

In the same way as Corollary 6.22 of \cite{MR1671260}, we obtain the Kasparov stabilization theorem for $(\phi,c,\tau)$-twisted $\G$-equivariant Hilbert $A$-modules when the groupoid $\G$ is proper. That is, for any $\sigma$-unital $\phi$-twisted $\G$-$\Cst$-algebra $A$ and countably generated $(\phi,c,\tau)$-twisted $\G$-equivariant Hilbert $A$-modules $E$, there is a $\G$-equivariant unitary equivalence 
\maa{ \lu \phi \Hilb _{\G,A}^{c,\tau} \oplus E \cong  \lu \phi \Hilb _{\G,A}^{c,\tau} \label{form:Kas}}
where $ \lu \phi \Hilb _{\G,A}^{c,\tau}$ denotes the $(\phi,c,\tau)$-twisted Hilbert $\G$-module $ \lu \phi \Hilb _{\G}^{c,\tau} \hotimes A$. In particular, $\Hilb _{\G ,A} \hotimes \mathscr{V} \cong \mathscr{H}_{\G,A}^{c,\tau}$ for any $(\phi,c,\tau)$-twisted representation $\mathscr{V}$ of $\G$.

\begin{defn}
Let $A$ and $B$ be $\Zt$-graded $\G$-$\Cst$-algebras. A \emph{$(\phi,c,\tau)$-twisted $\G$-equivariant Kasparov $A$-$B$ bimodule} is a triple $(E,\varphi,F)$ where
\begin{itemize}
\item a countably generated $(\phi,c,\tau)$-twisted $\G$-equivariant Hilbert $B$-module $E$,
\item a $\Zt$-graded $\ast$-homomorphism $\varphi:A \to \Lop (E)$ satisfying (\ref{form:superhom}),
\item an odd self-adjoint operator $F \in \Lop (E)$ such that $\varphi (a)F$ is $\G$-continuous, $[\varphi(a),F], \varphi (a)(F^2-1) \in \Kop (E)$ and $r^*\varphi (a)((-1)^{c(g)}\alpha (s^*F) -r^*F) \in \Kop (r^*E)$.
\end{itemize}
\end{defn}

We say that two $(\phi , c,\tau)$-twisted $\G$-equivariant Kasparov $A$-$B$ bimodules $(E_i,\varphi _i,F_i)$ are homotopic if there is a $(\phi , c,\tau)$-twisted $\G$-equivariant Kasparov $A$-$B[0,1]$ bimodule $(\tilde{E},\tilde{\varphi},\tilde{F})$ such that each $\ev_i ^*(\tilde{E},\tilde{\varphi},\tilde{F})$ is unitary equivalent to $(E_i,\varphi _i,F_i)$ for $i=0,1$.

\begin{defn}
We define $\lu{\phi} \KK ^\G_{c,\tau}(A,B)$ the group of homotopy classes of $(\phi ,c, \tau )$-twisted $\G$-equivariant Kasparov $A$-$B$ bimodules. 
\end{defn}

\begin{exmp}
For $\tau \in \check{H}^2(\G, \Zt)$, let $\G _\real$, $\phi _\real$ and $\tau _\real$ be as in Example \ref{exmp:real}. Then, the twisted equivariant $\KK$-group $\lu{\phi_\real} \KK ^{\G_\real}_{\tau_\real}(A,B)$ is isomorphic to the twisted equivariant Real $\KK$-group $\KKR ^{\G}(A,B)$ by definition. In the same way, the group $\lu {\phi_\real} \KK _{\tau_{\mathbb{H}}} ^{\G_\real} (A,B)$ is isomorphic to the twisted Quaternionic $\KK$-group $\KKQ ^{\G}(A,B)$. 
\end{exmp}

\begin{prp}\label{prp:detwist}
Let $A$ and $B$ be $\phi$-twisted $\Zt$-graded $\G$-$\Cst$-algebras. Set $\overline{\tau} := \tau +\epsilon (c,c)$. Then, there is a natural isomorphism
\[\lu{\phi}\KK^\G_{c,\tau}(A,B) \cong \lu \phi \KK^\G (A,B \hotimes \lu \phi \Kop _\G ^{c,\overline{\tau}}).\]
In particular, when $A$ is a Real $\G$-$\Cst$-algebra, 
\[
\lu{\phi}\KK^\G_{c,\tau}(A,B) \cong \KKR^\G (A,(B \hotimes \lu \phi \Kop _\G ^{c,\overline{\tau}})_\real).
\]
\end{prp}
It reduces the twisted equivariant $\KK$-group to the existing $\KK$-group since $\lu \phi \KK^\G (A,B)$ is isomorphic to $\KKR ^{\G \ltimes \G_\phi^0}(p_\phi^*A, p_\phi^*B)$ in the sense of Moutuou~\cite{MR3177819}. Here, the groupoid $\G \ltimes \G_\phi^0$ is regarded as a Real groupoid by the $\Zt$-action from the right. As a consequence, we obtain the $8$-fold Bott periodicity for twisted equivariant $\KK$-theory.
\begin{proof}
By Lemma \ref{lem:tensor} and Example \ref{exmp:Morita}, we have a one-to-one correspondence between $(\phi,c,\tau)$-twisted Kasparov bimodules for $A,B$ and $\phi$-twisted Kasparov bimodules for $A,B \hotimes \lu \phi  \Kop _\G^{c,\overline{\tau}}$ given by
\[
\renewcommand{\arraystretch}{1.3}
\begin{array}{ccc}
(E, \varphi , F) &\mapsto& (E \hotimes (\lu \phi \Hilb _{\G}^{c,\tau})^*, \varphi \hotimes 1, F \hotimes 1),\\
(E' \hotimes_{\lu \phi \Kop _\G^{c,\overline{\tau}}} \bdot{\Hilb} _{\G}^{c,\overline{\tau}} , \varphi' \hotimes 1, F' \hotimes 1)& \mapsfrom&  (E',\varphi',F').\\
\end{array}
\renewcommand{\arraystretch}{1}
\]
Note that the operators $F \hotimes 1$ and $F' \hotimes 1 $ satisfy the relations on group actions because $\Lop (E \hotimes (\lu \phi \Hilb _{\G}^{c,\tau})^*) \cong \Lop (\bdot{E})$ and $\Lop (E' \hotimes_{\lu \phi \Kop _\G^{c,\overline{\tau}}} \lu \phi \bdot{\Hilb} _{\G}^{c,\overline{\tau}}) \cong \Kop (\bdot{E'})$.

When $c=0$, $\tau=0$ and $A$ is a Real $\G$-$\Cst$-algebra, $\lu \phi \KK ^\G (A,B) \cong \lu \phi \KK^{\G _\phi}(p_\phi ^*A, p_\phi ^*B)$ is isomorphic to $\lu \phi \KK^{\G \times \Zt}(A,p_\phi ^*B)$, which is isomorphic to $\KKR ^\G(A,B_\real)$ by Remark \ref{remk:cpx}.
\end{proof}

At the end of this section, we briefly summarize basic properties of the Kasparov product. Let $(E_1,\varphi _1,F_1)$ be a $(\phi,\tau _1,c _1)$-twisted $\G$-equivariant Kasparov $A_1$-$B_1 \hotimes D$ bimodule and let $(E_2,\varphi _2,F_2)$ be a $(\phi,\tau _2,c _2)$-twisted $\G$-equivariant Kasparov $A_2 \hotimes D$-$B_2$ bimodule. Set $\tau = \tau _1 + \tau _2 + \epsilon (c_1,c_2)$ and $c=c_1+c_2$. We say that a $(\phi, c,\tau)$-twisted $\G$-equivariant Kasparov $A_1 \hotimes A_2$-$B_1 \hotimes B_2$ bimodule $(E ,\varphi , F)$ is a Kasparov product $(E_1,\varphi _1,F_1) \hotimes _B (E_2,\varphi _2,F_2)$ if $E=E_1 \hotimes _BE_2$, $\varphi (a_1 \hotimes a_2)= \varphi _1(a_1 ) \hotimes \varphi _2 (a_2)$ and $F$ satisfies
\maa{
\renewcommand{\arraystretch}{2}
\begin{array}{l}
\text{\parbox{.85\textwidth}{\hspace{-1.5ex}$\circ$ the operator $F$ is an $F_2$-connection, that is, $FT_\xi -(-1)^{\partial \xi}T_\xi F_2$ is in $\Kop (E_2,E)$ and $T_\xi^*F -(-1)^{\partial \xi}F_2T_\xi^*$ is in $\Kop (E,E_2)$ for any $\xi \in E_1$ (where $T_\xi(\eta):=\xi \hotimes \eta$).}}\\
\text{\parbox{.85\textwidth}{\hspace{-1.5ex}$\circ$ the operator $\varphi (a)[F,F_1 \hotimes 1]\varphi (a)^*$ is positive modulo compact for any $a \in A_1 \hotimes A_2$.}}
\end{array}
\renewcommand{\arraystretch}{1}
\label{form:KK}
}
It is straightforward to check that the Kasparov product always exists and induces the well-defined product of $\KK$-groups
\[
\lu \phi \KK ^\G_{c_1, \tau_1}(A_1,B_1 \hotimes D) \otimes \lu \phi \KK^\G_{c_2, \tau _2}(A_2 \hotimes D,B_2) \to \lu \phi \KK ^\G _{c,\tau}(A_1 \hotimes A_2,B_1 \hotimes B_2)
\]
which is associative by the same proof as Theorem 12 of \cite{MR743845} (for Real $\KK$-theory). 

As a corollary, we obtain an isomorphism between $\lu \phi \KK _{c,\tau }^\G (A,B)$ and the group $\lu \phi \KK _{c,\tau}^\G (A,B)_{\mathrm{oh}}$ of ``operator homotopy'' equivalence classes of $(\phi,c,\tau)$-twisted $\G$-equivariant Kasparov $A$-$B$ bimodules (cf.\ Subsection 18.5 of \cite{MR1656031}). Here, two $(\phi,c,\tau)$-twisted $\G$-equivariant Kasparov $A$-$B$ bimodules $(E_i,\varphi _i,F_i)$ are \emph{operator homotopic} if there are degenerate $(\phi,c,\tau)$-twisted $\G$-equivariant Kasparov $A$-$B$ bimodules $(E'_i,\varphi _i, F_i')$, an even unitary operator $U: E_1 \oplus E_1' \to E_2 \oplus E_2'$ and a norm continuous path $\tilde{F}_t$ of operators on $E_2 \oplus E_2'$ such that $(E_2 \oplus E_2', \varphi _2 \oplus \varphi _2', \tilde{F}_t)$ are $(\phi,c,\tau)$-twisted $\G$-equivariant Kasparov $A$-$B$ bimodules and $\tilde{F}_0=U^* (F_1 \oplus F_1')U$, $\tilde{F}_1=F_2 \oplus F_2'$.

\section{Twisted equivariant $\K$-theory}\label{section:4}
In the same fashion as the untwisted case, it is natural to define the twisted equivariant $\K$-group as a special case of the twisted equivariant $\KK$-groups introduced in Section \ref{section:3}. In this section, we study basic properties of twisted equivariant $\K$-groups. In particular, we introduce a generalization of the twisted crossed product for twists in the sense of \cite{MR3119923} and prove the generalization of the Green-Julg theorem.

Hereafter, we assume that $\G$ is a proper groupoid.

\begin{defn}\label{def:twK}
Let $A$ be a $\phi$-twisted $\Zt$-graded $\G$-$\Cst$-algebra. The twisted equivariant $\K$-group $\lu{\phi} \K _{0,c,\tau}^\G(A)$ is defined as the twisted equivariant $\KK$-group $\lu{\phi} \KK ^\G_{c,\tau}(\real , A)$.
\end{defn}

We remark that every element in $\lu{\phi} \KK ^G_{c,\tau}(\real , A)$ is represented by a Kasparov bimodule of the form $[\lu \phi \Hilb _{\G,A}^{c,\tau},1,F]$. Actually, a Kasparov bimodule $(E,\varphi,F)$ is equivalent to $(pE \oplus  \lu \phi \Hilb _{\G ,A}^{c,\tau},p,pFp \oplus F_0)$ and $pE \oplus \lu \phi \Hilb _{\G ,A}^{c,\tau} \cong \lu \phi \Hilb _{\G ,A}^{c,\tau}$ by (\ref{form:Kas}), where $p :=\varphi (1)$ and $(\lu \phi \Hilb _{\G,A}^{c,\tau}, \varphi _0, F_0)$ is a degenerate Kasparov bimodule such that $\varphi _0$ is unital.

This definition has another reasonable presentation, which is a generalization of $\K$-theory for $\Zt$-graded $\Cst$-algebras introduced in \citelist{\cite{MR2058474}\cite{MR1775323}\cite{MR1487204}}. Let $\Salg$ be the $\Cst$-algebra $C_0(\real)$ together with the $\Zt$-grading $\gamma (f)(x)=f(-x)$ and the Real structure $\overline{f}(x)=\overline{f(x)}$. The groupoid $\G$ acts on $\mathcal{S}$ by $\alpha _g(f)=\gamma ^{c(g)} (\lu{\phi (g)}f)$.

\begin{prp}\label{prp:twK}
Let $A$ be a $\phi$-twisted $\Zt$-graded $\G$-$\Cst$-algebra. Then, the group $\lu \phi \K _{0,c,\tau}^\G (A)$ is isomorphic to the set of homotopy classes of $\G$-equivariant $\ast$-homomorphisms $[\Salg , A \hotimes  \lu \phi \Kop_{\G}^{c,\tau}]^\G$.
\end{prp}
For a $\G$-$\Cst$-algebra $A$, let $A^+$ denote the unitization $A + C_b(\G ^0) \subset M(A \oplus \real)$ of $A$. Then, there is a one-to-one correspondence between $\G$-equivariant $\ast$-homomorphisms from $\Salg $ to $A \hotimes \lu \phi \Kop_{\G}^{c,\tau}$ and unitaries $(A \hotimes \lu \phi \Kop_{\G}^{c,\tau})^+$ such that $u-1 \in A \hotimes \lu \phi \Kop _\G^{c,\tau}$, $\gamma (u)=u^*$ and $\alpha _g(u)=\gamma ^{c(g)+\phi (g)}(u)$. Actually, a $\ast$-homomorphism $\varphi$ corresponds to a unitary $1+\varphi (-\exp (\pi ix(1+x^2)^{-1})-1)$.
\begin{proof}
The proof is given in the same way as in Theorem 4.7 of \cite{MR1775323}. By the functional calculus given in Theorem 3.2 of \cite{MR1775323}, we obtain a $\Zt$-graded submodule $E:=\overline{\varphi (C_c(\real)) \lu{\phi}\Hilb _{\G,A}^{c,\tau}}$ of $ \lu \phi \Hilb _{\G,A}^{c,\tau}$ and a regular odd self-adjoint operator $D$ on $E$ given by $D(\varphi(f) \xi) := \varphi (xf) \xi$. Consequently we obtain a Kasparov $\real$-$(A \hotimes \lu \phi \Kop _\G ^{c,\tau })$ bimodule $[E,1 ,D(1+D^2)^{-1/2}]$. This correspondence gives a group homomorphism $[\Salg ,A \hotimes \lu \phi \Kop _\G ^{c,\tau }]^\G \to \lu \phi \KK ^\G_{c,\tau}(\real, A)$. 

In the same way as the proof of Theorem 4.7 of \cite{MR1775323}, we can check that it is an isomorphism. 
We remark that for any Kasparov bimodule $(E,\varphi,F)$, we can replace $F$ with another Fredholm operator $F'=\int _{g \in \G ^x}\alpha_g (F_{s(g)})d\lambda ^x(g)$ commuting with the $\G$-action since $\G$ is proper.
\end{proof}

Now we generalize the Green-Julg theorem \cite{MR625361} for twisted equivariant $\K$-theory.  First, we start with the definition of the crossed product for groupoid actions with general twists. Here, we say that an extension of a groupoid $\G$ by a group $N$ is a principal $N$-bundle $\G ^\sigma$ over $\G ^1$ with a groupoid structure which is compatible with the projection $\pi : \G ^\sigma \to \G$ and the identification $\pi ^{-1}(1_x) =1_x \cdot N \cong N$ for each $x \in \G ^0$. Denote $n_x := 1_x \cdot n$ for $n \in N$.
\begin{defn}
Let $\phi : \G \to \Zt$ be a groupoid homomorphism and let $A$ be a $\G ^0$-$\Cst$-algebra. We say that a $(\phi,\sigma)$-twisted $\G$-action is a pair $(\G^\sigma, \alpha )$ where $\G ^\sigma$ is a groupoid extension of $\G$ by a locally compact subgroup $N$ of $\mathcal{U}M(A)$ and $\alpha$ is a $\phi$-twisted action of $\G^\sigma$ on $A$ such that $\alpha _{u_x} =\Ad u_x$ for any $u \in N$.  
\end{defn}
In the same way as in Section \ref{section:3}, we get a family of $\phi$-linear $\G^0$-$\ast$-automorphisms $\alpha =\{ \alpha ^\mu : s^*A|_{U^\mu} \to r^*A|_{U^\mu} \}_{\mu \in I^1}$ and a family $\sigma =\{ \sigma ^\nu : V^\nu \to N \}_{\nu \in I^2}$ such that 
\ma{
\alpha _g^{\tilde{\varepsilon} _0 \nu} \circ \alpha _h^{\tilde{\varepsilon} _2 \nu} &=\Ad (\sigma ^{\nu}(g,h)_{r(g)}) \circ \alpha _{gh}^{\tilde{\varepsilon} _1 \nu}, \\
\sigma ^{\tilde{\varepsilon} _1\kappa } (gh,k)\sigma ^{\tilde{\varepsilon} _3 \kappa }(g,h)&=\alpha _g^{\tilde{\varepsilon} _2\tilde{\varepsilon} _1\kappa}(\sigma ^{\tilde{\varepsilon} _0 \kappa}(h,k))\sigma ^{\tilde{\varepsilon} _2 \kappa} (g,hk), \\
\sigma^{\nu} (g,1_{r(g)})&=\sigma^{\nu} (1_{s(g)},g)=1,\\
\alpha _{1_x}^\mu &=\id_{A_x}, }
for each $\nu \in I^2$, $\kappa \in I^3$, $x \in \G ^0$ and $g \in \G ^1$. We say that $\sigma$ is a $2$-cocycle associated to $\alpha$.

\begin{defn}
Let $c:\G \to \Zt$ be a groupoid homomorphism. For a $(\phi,\sigma)$-twisted $\G$-action on a $\Cst$-algebra $A$, the \emph{$(\phi,c,\sigma)$-twisted crossed product} $\G \ltimes _{c,\sigma }^\phi A$ is the $\Zt$-graded $\Cst$-algebra $\G \ltimes _{c,\sigma}A_\real$ together with the Real structure induced from the complex conjugation on $A_\real$. 
\end{defn}

Here, for a $\sigma$-twisted $\G$-$\Cst$-algebra, the full graded twisted crossed product $\G \ltimes _{c,\sigma}A$ (see Section 2 of ~\cite{MR1857079}) is given by the completion of the subalgebra $C_c(\G, r^*A, \sigma)$ of $C_b(\G^\sigma, r^*A)$ consisting of functions $f$ with compact support on $\G$ such that $f(\tilde{g}u)=f(\tilde{g}) \alpha _{\tilde{g}}(u^*) \in A_{r(\tilde{g})}$
with operations
\ma{
f_1 \ast f_2 (\tilde{g})&:=\int _{\G^{r(\tilde{g})}} (-1)^{c(\tilde{g})\cdot |f_1(\tilde{h})|} f_1(\tilde{h})\alpha _{\tilde{h}}^A(f_2(\tilde{h}^{-1}\tilde{g})) d\lambda ^{r(\tilde{g})}(h),\\
f^*(\tilde{g})&:=(-1)^{c(\tilde{g}) \cdot |f(\tilde{g})|} \alpha _{\tilde{g}}^A (f(\tilde{g}^{-1})^*), \\
\alpha _{\tilde{h}}(f)(\tilde{g})&:=\alpha _{\tilde{h}}^A(f(\tilde{h}^{-1}\tilde{g})), \\
\gamma (f)(\tilde{g})&:=(-1)^{c(\tilde{g})}\gamma _A(f (\tilde{g})),
}
for $\tilde{g}, \tilde{h} \in \G^\sigma$ by the maximal $\Cst$-norm smaller than
\[
\ssbk{f}_1:=\sup _{x \in \G^0} \max \left\{ \int _{\G ^x}\ssbk{f(g)}d\lambda ^x (g), \ \int _{\G _x}\ssbk{f(g)}d\lambda _x (g) \right\}
\]
(note that $\ssbk{f(\tilde{g})}$ is a function on $\G$). Note that our definition is slightly different from Definition 5.1 of \cite{mathph14067366}.

In Section 3.6 of \cite{MR2074181}, the reduced crossed product of a groupoid action is defined as the quotient of the full twisted crossed product with respect to the family of kernels of (pointwise) regular representations. The difference between full and reduced crossed products causes no problem for us because they always coincide when $\G$ is proper. 

For a $(\phi , c)$-twisted representation $u$ of $\G^\sigma$ and a $\ast$-representation $\pi$ of $A$ on a Hilbert bundle $\Hilb$ over $\G ^0$ which are twisted covariant (that is, it satisfies $u_g \pi (a) u_g^*=(-1)^{c(g)\cdot |a|}\pi  (\alpha _g (a))$ for any $g \in \G^1$ and $a \in A_{s(g)}$), we have a unique representation $\tilde{\pi}$ of the full crossed product $\G \ltimes ^\phi _{c,\tau} A$ on $p_\phi^*\Hilb \cong \Hilb _\real$. 

\begin{exmp}
When $\phi$ is trivial, then the $(\phi,c,\sigma)$-twisted crossed product is isomorphic to $(\G \ltimes _{c,\sigma}A)_\real$, whose Real $\K$-group is isomorphic to the $\K$-group of $\G \ltimes _{c,\sigma} A$. 
\end{exmp}

\begin{exmp}\label{exmp:triv}
When $\G$ acts on the Real $\Cst$-algebra $\real$ trivially, the crossed product $\G \ltimes _{c,\tau}^\phi \real $ is called the \emph{$(\phi,c,\tau)$-twisted groupoid $\Cst$-algebra} and denoted by $\lu \phi C^*_{c,\tau}\G$. For example, the twisted groupoid $\Cst$-algebra of $\G _\real$ with respect to a real twist $(\phi _\real,c, \tau _\real )$ is isomorphic to $\Mop _2 (C^*_{c,\tau}(\G;\real))$ where $C^*_{c,\tau}(\G;\real)$ is the twisted graded groupoid $\Cst$-algebra with coefficient in $\real$. For a general $\Zt$-valued $1$-cocycles $c$ on $\G$, we define the \emph{stable twisted groupoid $\Cst$-algebra} as $\lu \phi \mathring{C}^*_{c,\tau}\G := \lu \phi C^*_{c,\tau}\G_c$. When $c$ is represented by a groupoid homomorphism, it is isomorphic to $\hat{\Mop }_{1,1} (\lu \phi C^*_{c,\tau}\G)$. For example, when $(\mathscr{A}, \tau )$ is a CT-type as in Example \ref{exmp:CT}, then $\lu \phi \mathring{C}^*_{c,\tau}\mathscr{A}$ is isomorphic to a Clifford algebra as indicated in Table \ref{table:CT} below. 
\end{exmp}

\begin{table}[h]
\scalebox{0.92}[0.92]{
\begin{tabular}{|r||c|c||c|c|c|c|c|c|c|c|} \hline
$\mathscr{A}$ & $1$ & $\mathscr{P}$ & \multicolumn{2}{c|}{$\mathscr{T}$} & \multicolumn{2}{c|}{$\mathscr{C}$} & \multicolumn{4}{c|}{$\mathscr{G}$}\\ \hline
$C^2$& \diaghead{$\Cliff _{4}^2$}{}{} & \diaghead{$\Cliff _{4,2}^2$}{}{} & \diaghead{$\Cliff _{4,2}^2$}{}{} & \diaghead{$\Cliff _{4,2}^2$}{}{} & $1$& $-1$ & $1$ & $1$ & $-1$ & $-1$ \\ \cline{1-3}\cline{4-11}
$T^2$& \diaghead{$\Cliff _{4}^2$}{}{} &\diaghead{$\Cliff _{4,2}^2$}{}{}& $1$ & $-1$ & \diaghead{$\Cliff _{4,2}^2$}{}{}  & \diaghead{$\Cliff _{4,2}^2$}{}{} & $1$ & $-1$ & $1$ & $-1$ \\ \hhline{|=#=|=#=|=|=|=|=|=|=|=|}
$\lu \phi \mathring{C}^*_{c,\tau}\sA$& $\Cliff _2$ & $\Cliff _3$ & $\Cl _{2,2}$ &$\Cl _{0,4}$ & $\Cl _{1,3}$ & $\Cl_{3,1}$ & $\Cl_{2,3}$ & $\Cl_{1,4}$ & $\Cl_{3,2}$ & $\Cl_{4,1}$ \\ \hline 
$\lu \phi \K _{0,c,\tau}^\sA$ & $\K _{0}$ & $\K _{1}$ & $\KR _0$ & $\KR _{4}$ & $\KR _{2}$ & $\KR _{6}$ & $\KR _{1}$ & $\KR _{3}$ & $\KR_{7}$ & $\KR _{5}$ \\ \hline
Cartan&A & AI\hspace{-.1em}I\hspace{-.1em}I  & AI & AI\hspace{-.1em}I  & D & C & BDI & DI\hspace{-.1em}I\hspace{-.1em}I  & CI  & CI\hspace{-.1em}I \\ \hline
\end{tabular}
}
\caption{The 10-fold way and Clifford algebras}
\label{table:CT}
\end{table}

\begin{exmp}
Let $\G _\real$, $\phi _\real$ and $\tau _\real$ be as in Example \ref{exmp:real}. For a Real $\G$-$\Cst$-algebra $A$ and $c \in \Hom (\G, \Zt)$, the twisted crossed product $\G _{\real} \ltimes_{c,\tau_\real} ^{\phi_\real} A$ is isomorphic to $\Mop _2(\G \ltimes _{c,\tau} A)$. In the same way, the crossed product $\G_\real \ltimes ^{\phi_\real}_{c, \tau_{\mathbb{H}}} A$ is isomorphic to the quaternionic crossed product $\G \ltimes _{c,\tau}^{\mathbb{H}} A:=(\G \ltimes _{c,\tau} A) \otimes_\real \mathbb{H}$. 
\end{exmp}

\begin{exmp}\label{exmp:Cl2}
Let $F_{n,m}$ and $(c_{n,m}^0, \tau _{n,m}^0)$ be as in Example \ref{exmp:Cl}. The real graded group $\Cst$-algebra $C^*_{c_{n,m}^0,\tau_{n,m}^0}(F_{n,m})$ is isomorphic to the complex Clifford algebra $\Mop _2 (\Cliff _{n+m})$. Moreover, let $\G$ be a groupoid, let $(\phi,c,\tau)$ be a twist on $\G$ and let $A$ be a $\phi$-twisted $\G$-$\Cst$-algebra. Set $\G'=\G _{n,m}:=\G \times F_{n,m}$ and 
\[(\phi,c',\tau')=(\phi , c_{n,m}, \tau _{n,m}):= (\phi \times 0, c \times c_{n,m}^0, \tau \times \tau _{n,m}^0).\]
Then, $ \G' \ltimes ^{\phi}_{c', \tau'} A$ is isomorphic to $(\G \ltimes ^\phi _{c,\tau} A) \hotimes \Cl _{n,m}$.
\end{exmp}

\begin{lem}\label{lem:cross}
The crossed product $\G \ltimes _{c,\tau}^\phi A$ is Morita equivalent to $\G \ltimes (A \hotimes \lu{\phi} \Kop _\G^{c,-\tau})_\real$, where  as Real $\Cst$-algebras.
\end{lem}
\begin{proof}
First, we reduce the problem for the case that $c=0$. Let $\mathscr{V}$ be the $\Zt$-graded vector space $\hat{\real}$ with the Real $c$-twisted representation of $\G$ given by $\Ad (1 \hotimes e)^{c(g)}$ where $e:=\pmx{0 & 1 \\ 1 & 0}$. Then, we have a $\ast$-isomorphism
\[ \G \ltimes ^\phi_\tau (A \hotimes \Kop (\mathscr{V})) \to \G \ltimes _{c,\tau}^\phi (A \hotimes \Kop (\hat{\real}))  \cong (\G \ltimes _c^\phi A) \hotimes \Kop (\hat{\real}) \]
given by $f(\tilde{g}) \mapsto (1 \hotimes e\gamma)^{c(\tilde{g})}f(\tilde{g})$ for any $f \in C_c(\G,r^*A,\tau)$.

Next, we assume that $c=0$. Recall that two $\Cst$-algebras $A$ and $B$ are Morita equivalent if and only if there is a $\Cst$-algebra $D$ and a projection $p \in M(D)$ such that $pDp \cong A$, $(1-p)D(1-p)\cong B$ and $DpD:=\mathrm{span} \{ d_1pd_2 \mid d_i \in D \}$ is dense in $D$. (See \cite{MR0463928}. Actually, $pD(1-p)$ is a $A$-$B$ imprimitivity bimodule and conversely $D:=\pmx{A & E \\ E^* & B}$ satisfies the above properties.)

Let $B$ be the $\Cst$-algebra $A \hotimes \Kop (\lu \phi \Hilb _\G ^{-\tau} \oplus \real) = \pmx{A \hotimes  \Kop ^{-\tau}_\G & A \hotimes  (\lu \phi \Hilb^{-\tau}_\G)^* \\ A \hotimes \Hilb ^{-\tau}_\G & A}$ with the $(\phi ,\sigma)$-twisted $\G$-action
\[
\alpha _g (\pmx{a \hotimes x & b \hotimes \xi^* \\ c \hotimes \eta & d})=\pmx{\alpha _g (a) \hotimes u_gxu_g^* & \alpha _g (b) \hotimes (u_g \xi)^* \\ \alpha _g (c)\hotimes u_g\eta & \alpha _g(d) }
\]
where $\sigma$ is the $\mathcal{U}M(B)$-valued $2$-cocycle given by $\sigma (g,h)= \pmx{1 & 0 \\ 0 &  \tau (g,h)}$. Set $D:=\G \ltimes _{\sigma }B$. Then, the projection $p:=\pmx{1 & 0 \\ 0 & 0}$ in $M(D)$ satisfies $pDp \cong \G \ltimes^\phi (A \hotimes \lu \phi  \Kop _G^{-\tau})$ and $(1-p)D(1-p) \cong \G \ltimes _{\tau }^\phi A$. Moreover, since $D$ contains a dense subspace $C_c(\G)B$ and $BC_c(\G)$, $DpD$ contains a dense subset $C_c(\G)BpBC_c(\G) =C_c(\G)BC_c(\G)$ in $D$. 
\end{proof}

\begin{thm}\label{thm:GJ}
Let $\G$ be a proper groupoid. Then there is an isomorphism
\[ \lu{\phi}\K ^\G_{*,c,\tau}(A)\cong \KR_*( \G \ltimes _{c,-\overline{\tau}}^\phi A),\]
where $\overline{\tau}:=\tau + \epsilon (c,c)$.
\end{thm}
\begin{proof}
The Green-Julg theorem is generalized for proper groupoids in Proposition 6.25 of \cite{MR1671260}. By Proposition \ref{prp:detwist}, $\lu{\phi}\KK^\G_{c,\tau} (\real , A)$ is isomorphic to the Real $\KK$-group $ \KKR ^{\G}(\real , (A \hotimes \lu \phi \Kop _\G^{c,\overline{\tau}})_\real )$. Now the conclusion follows from the usual Green-Julg theorem for Real $\Cst$-algebras and Lemma \ref{lem:cross}. 
\end{proof}

\begin{cor}\label{cor:Cl}
Let $\G':=\G_{n,m}$ and $(\phi,c',\tau'):=(\phi, c_{n,m},\tau _{n,m})$ be as in Example \ref{exmp:Cl2}. For any $\phi$-twisted $\G$-$\Cst$-algebra $A$, we have an isomorphism
\[
\lu {\phi} \K ^{\G'} _{0,c',\tau'}(A) \cong \lu \phi \K ^\G _{m-n,c,\tau}(A).
\]
In particular, $\lu{\phi_\real}\K ^{\G'_\real} _{0,c',\tau'_\real}(A) \cong \KR ^\G _{m-n,c,\tau}(A)$ for any Real $\G$-$\Cst$-algebra $A$.
\end{cor}
\begin{proof}
It follows from Example \ref{exmp:triv}, Example \ref{exmp:Cl2} and the Green-Julg theorem \ref{thm:GJ}. Note that $\tau _{n,m}+\epsilon (c_{n,m},c_{n,m})=\tau _{m,n}$ for Clifford twits.
\end{proof}

\begin{cor}\label{cor:CT}
Let $(\G,\phi,c,\tau)$ be a CT-type symmetry as in Example \ref{exmp:CT} with the CT-type $(\sA,\tau)$ and let $A$ be a $\Zt$-graded Real $\G_0$-$\Cst$-algebra. Then, the twisted equivariant $\K$-group $\lu \phi \K _{*,c,\tau}^\G (A)$ is isomorphic to the equivariant $\K$-group $\KF ^{\G_0}_{n+*}(A)$. Here, $\KF _n$ is one of $\K_n$ ($n=0,1$) or $\KR _n$ ($n=0,\dots,7$) as indicated in Table \ref{table:CT}.
\end{cor}
\begin{proof}
It follows from the Green-Julg theorem \ref{thm:GJ} and Example \ref{exmp:triv}. Since $\mathscr{A}$ acts on $A$ trivially, the crossed product $\G \ltimes _{c,\tau}^\phi A$ is isomorphic to $(\G_0 \ltimes _{\tau _0}A ) \hotimes (\lu \phi C^*_{c,\tau}\mathscr{A})$. 
\end{proof}

\section{The group ${}^\phi \K_{1,c,\tau}^G(A)_{\vD}$}\label{section:5}
In this section, we introduce another formulation of twisted equivariant $\K$-theory based on van Daele's definition of $\K$-theory for $\Zt$-graded Banach algebras. This definition is related to Karoubi's $\K$-theory \citelist{\cite{MR0238927}\cite{MR2458205}} in the case that the $\Cst$-algebra is trivially graded.

Throughout this section, we assume that $\G$ is a proper groupoid with the compact orbit space such that $\lu \phi \Hilb_{\G}^{c,\tau}$ is AFGP for any $(\phi, c,\tau)$. Here we say that a (twisted) representation $\Hilb$ is AFGP (approximately finitely generated projective) if there is a $\G$-invariant approximate unit of projections $(p_n)$ in $\Kop (\Hilb)$ (Definition 5.14 of \cite{MR2119241}). By Lemma 3.1 of \cite{MR1623142}, we can replace it with an increasing approximate unit of projections]. In other words, $\Hilb$ is AFGP if and only if it is decomposed into a direct sum of finite dimensional representations. 

\begin{remk}\label{rmk:decompble}
A reasonable criterion for $\Hilb_{\G}$ to be AFGP is given in Theorem 6.14 of \cite{MR2499440}; it holds if and only if every irreducible representation of $\G _x^x$ is contained in the fiber of a finite dimensional representation of $\G$ (see also Theorem 6.15 of \cite{MR2499440}).  
For example, this condition is satisfied for a translation groupoid $G \ltimes X$ if $G$ is compact. 

Note that $\Hilb_\G^{c,\tau}$ is AFGP if and only if so is $\Hilb _\G$ and there is a finite dimensional $(\phi,c,\tau)$-twisted representation of $\G$ since $\lu \phi \Hilb_\G^{c,\tau} \cong \Hilb _\G \hotimes \mathscr{V}$ (cf.\ Proposition 5.25 of \cite{MR2119241}). 
It is conjectured in Subsection 5.7 of \cite{MR2119241} that $\G$ has a $\tau$-twisted representation if and only if $\tau$ is in the image of $\check{H}^2(\G ; (\quot /\zahl)_\phi ) \to \check{H}^2(\G ; \T _\phi)$. 
It is known to be true in the case that $\G$ is a compact group ($\G$ always has a $\tau$-twisted $\G$-vector bundle), $\G$ is a compact space (the Grothendieck-Serre theorem \citelist{\cite{MR1608798}\cite{MR2172633}}) and $\G$ is a Renault-Deaconu groupoid \cite{mathOA150500364} but still open in general. 
\end{remk}
 
Let $A$ be a unital $\phi$-twisted $\Zt$-graded $\G$-$\Cst$-algebra. We say that an element $a \in A$ is \emph{graded $\G$-invariant} with respect to the $\Zt$-grading $c$ on $\G$ if $\alpha _g (a)=\gamma^{c(g)}(a)$ for any $g \in \G$. Let $\lu{\phi}F _c^\G(A)$ denote the set of graded $\G$-invariant odd symmetries of $A$ (remark that $\lu{\phi}F_c^\G(A)$ is possibly empty, for example when $A$ is trivially graded). 

For a finite dimensional $(\phi, c, \tau)$-twisted unitary representation $\mathscr{V}$ of $\G$, set $\lu \phi F _{c,\mathscr{V}}^\G (A):=\lu \phi F_c^\G (A \hotimes \Kop (\mathscr{V}))$. For a fixed element $e \in \lu \phi F_c^\G(A)$, set $\lu \phi \bK _{c,\tau}^\G(A)_e:=\varinjlim _\mathscr{V} \pi _0 \lu \phi F_{c,\mathscr{V}}^\G (A)$, where $\mathscr{V}$ runs over all finite dimensional $(\phi,c,\tau)$-twisted unitary representations. 
Here, the inductive limit is taken with respect to the map $\pi _0 F_{c,\mathscr{V}}^\G (A) \to \pi _0F_{c,\mathscr{W}}^\G(A)$ for $\mathscr{V} \subset \mathscr{W}$ given by $x \mapsto x \oplus e \hotimes 1_{\mathscr{V}^\perp}$. 
The direct sum $(x,y) \mapsto x \oplus y$ determines a map from $\lu \phi F_{c,\mathscr{V}}^\G (A) \times \lu \phi F_{c,\mathscr{W}}^\G (A)$ to $\lu \phi F_{c,\mathscr{V} \oplus \mathscr{W}}^\G (A)$, which induces a map $\lu \phi \bK _{c,\tau}^\G(A)_e \times \lu \phi \bK _{c,\tau}^\G(A)_e \to \lu \phi \bK _{c,\tau}^\G(A)_e$. 
Note that we obtain the same group by taking limits on finite direct sums of $\Hilb_\G$.
By the same reasoning as Proposition 2.7 of \cite{MR947500}, it is associative and commutative. Consequently $\lu \phi \bK _{c,\tau}^\G(A)_e$ has the structure of abelian semigroup with the unit $[e]$.

\begin{prp}\label{prp:indep}
Assume that there is an element $e$ in $\lu{\phi}F_{c}^\G(A)$ and an even $\G$-invariant unitary $v \in A$ such that $vev^*=-e$. Then $\lu \phi \bK_{c,\tau}^\G(A)_e$ is an abelian group and independent of the choice of such $e$ up to isomorphism.
\end{prp}
\begin{proof}
The proof is given in the same way as in Proposition 3 of \cite{MR994490}. Actually, if there is another $f \in \lu{\phi}F^\G_{c,\tau}(A)$ and even $\G$-invariant unitary $w$ such that $wfw^*=-f$, the map 
\[
\Ad \bk{ \frac{1}{2}\pmx{1 & 0 \\ 0 & w^*}\pmx{1+fe & 1-fe \\ 1-fe & 1+fe}\pmx{1 & 0 \\ 0 & v}}
\]
induces an isomorphism $\lu \phi \mathbf{K}_{c,\tau}^\G(A)_e \to \lu \phi \mathbf{K}_{c,\tau}^\G(A)_f$.
\end{proof}
\begin{remk}\label{rmk:indep}
Let $u$ be an even $\G$-invariant unitary in $A \hotimes \Kop (\mathscr{V})$ commuting with $e$. Then, $\Ad u$ acts trivially on $\lu \phi \bK_{c,\tau}^\G(A)_e$ since $usu^* \oplus e = (u \oplus u^*)(s \oplus e)(u \oplus u^*)$ and there is a homotopy $(u \oplus 1) R_t (u^* \oplus 1)R_t$ of even $\G$-invariant unitaries commuting with $e \oplus e$ connecting $u \oplus u^*$ and $1$ (where $R_t$ are the $2$-by-$2$ rotation matrices). Consequently, the isomorphism in Proposition \ref{prp:indep} is independent of the choice of unitaries $v$. Moreover, when $e$ and $f$ are anticommutative, we have a simple presentation $\frac{1}{2}\Ad (1+ef)$ of the isomorphism $\lu \phi \bK_{c,\tau}^\G(A)_e \cong \lu \phi \bK_{c,\tau}^\G(A)_f$. 
\end{remk}

Let $\hat{\real} ^{2,2}_c$ be the $\Zt$-graded real Hilbert space $\real \oplus \real ^{\mathrm{op}} \oplus \real \oplus \real ^{\mathrm{op}}$ with the $\phi$-twisted representation $u_g \xi =u^{c(g)} (\lu{\phi(g)}\xi)$ of $\G$. Then, the $\Zt$-graded $\phi$-twisted $\Cst$-algebra $\hat{\Mop} _{2,2}^c:=\Kop (\hat{\real} ^{2,2}_c)$ is Morita-equivalent to $\real$ and has a graded $\G$-invariant odd symmetry $v$. Here, 
\[u=\pmx{0 & 0 & 1 & 0 \\ 0 & 0 & 0 & 1 \\ 1 & 0 & 0 & 0 \\ 0 & 1 & 0 & 0}, v=\pmx{0 & 1 & 0 & 0 \\ 1 & 0 & 0 & 0 \\ 0 & 0 & 0 & -1 \\ 0 & 0 & -1 & 0}.\]
We simply write $\hat{\Mop} _{2,2}^c(A)$ for the tensor product $\hat{\Mop} _{2,2}^c \hotimes A$.

\begin{defn}
The twisted equivariant $\K_1$-group $\lu{\phi}\K _{1,c,\tau}^\G(A)_{\vD}$ is defined by the abelian group $\Ker (\lu{\phi}\bK_{c,\tau}^\G(\hat{\Mop} _{2,2}^c(A^+))_{v} \to \lu{\phi}\bK_{c,\tau}^\G(\hat{\Mop} _{2,2}^c))_{v}$.
\end{defn}

We remark that when $A$ is unital $A^+$ is isomorphic to the direct sum $A \oplus C_b(\G ^0)$ and hence $\lu{\phi}\K ^\G_{1,c,\tau}(A)_{\vD}$ is isomorphic to $\lu{\phi}\bK_{c,\tau}^\G(\hat{\Mop}_{2,2}^c(A))_{v}$.

\begin{lem}\label{lem:vDdetw}
There is a natural isomorphism $\lu{\phi}\K _{1,c,\tau}^\G(A)_{\vD} \cong \lu{\phi}\K _{1}^\G(A \hotimes \lu \phi \Kop _{\G}^{c,\overline{\tau}})_{\vD}$, where $\overline{\tau} :=\tau + \epsilon (c,c)$.
\end{lem}
\begin{proof}
For a $(\phi,c,\tau)$-twisted representation $\mathscr{V}$, let $\dot{\mathscr{V}}$ be as in Example \ref{exmp:Morita}.
The identity $\ast$-homomorphism $A \hotimes \Kop (\mathscr{V}) \to A \hotimes \Kop (\bdot{\mathscr{V}})$ gives a homeomorphism $\varphi_* : \lu \phi F_{c}(A \hotimes \Kop (\mathscr{V})) \to \lu \phi F(A \hotimes \Kop (\bdot{\mathscr{V}}))$. Since $\Kop (\bdot {\mathscr{V}})$ is Morita equivalent to $\lu \phi \Kop _{\G}^{c,\overline{\tau}}$, the inductive limit $\varinjlim _{\mathscr{V}}\pi _0  \lu \phi F(A \hotimes \Kop (\bdot{\mathscr{V}}))$ is isomorphic to $\lu{\phi}\K _{1}^\G(A \hotimes \lu \phi \Kop _{\G}^{c,\overline{\tau}})_{\vD}$. Therefore, $\varphi _*$ gives a desired isomorphism.
\end{proof}

\begin{lem}\label{lem:sym}
If $s,t \in \lu \phi F_{c,\mathscr{V}}^\G(A)$ satisfies $\ssbk{s-t} <2$, there is a $\G$-invariant even self-adjoint element $h \in A \hotimes \Kop (\mathscr{V})$ such that $\ssbk{h}<1$, $tht=-h$ and $s=\exp (2\pi i h)t$. Moreover, $s$ and $t$ are homotopic in $\lu \phi F_{c,\mathscr{V}}^\G(A)$.
\end{lem}
\begin{proof}
Since $\ssbk{st-1}<2$, $-1$ is not contained in the spectrum of the $\G$-invariant even unitary $st$. Hence $h:=\log st$ with respect to the branch $\{ -i\real _{\geq 0} \}$ is well-defined self-adjoint element in $A \hotimes \Kop (\mathscr{V})$ such that $\ssbk{h}<1$ and $tht=\log (t(st)t)=\log (st)^*=-\log st=-h$. Now, $\tau \mapsto \exp (2 \pi i \tau h)t$ gives a continuous path connecting $s$ and $t$ in $\lu \phi F_{c,\mathscr{V}}^\G(A)$.
\end{proof}

\begin{lem}\label{lem:vDisK}
The functor $\lu{\phi}\K ^\G_{1,c,\tau}(\blank)_\vD$ has the following properties.
\begin{enumerate}
\item The correspondence $A \mapsto \lu{\phi}\K^\G_{1,c,\tau}(A)_\vD$ is a covariant functor from the category of $\phi$-twisted $\Zt$-graded $\G$-$\Cst$-algebras and $\ast$-homomorphisms to the category of abelian groups.
\item If two $\G$-equivariant $\ast$-homomorphisms $\varphi$ and $\psi$ are homotopic, then $\varphi _* =\psi _*:\lu{\phi}\K^\G_{1,c,\tau} (A)_\vD \to \lu{\phi}\K^\G_{1,c,\tau}(B)_\vD$.
\item For any $\phi$-twisted representations $\mathscr{V}_1 \subset \mathscr{V}_2$ of $\G$, the inclusion $\Kop (\mathscr{V}_1) \to \Kop (\mathscr{V}_2) $ induces the isomorphism 
\[
f_{\mathscr{V}_1, \mathscr{V}_2}^A: \lu \phi \K^\G_{1,c,\tau}(A \hotimes \Kop (\mathscr{V} _1))_\vD \to \lu \phi \K ^\G_{1,c,\tau}(A \hotimes \Kop (\mathscr{V}_2) )_\vD .
\]
\item A short exact sequence $0 \to I \xra{\iota} A \xra{\pi} A/I \to 0$ of $\phi$-twisted $\G$-$\Cst$-algebras induces an exact sequence
\[
\lu{\phi}\K _{1,c,\tau}^\G(I)_\vD \to \lu{\phi}\K_{1,c,\tau}^\G (A)_\vD \to \lu{\phi}\K_{1,c,\tau}^\G (A/I)_\vD.
\]
\end{enumerate}
\end{lem}
\begin{proof}
The assertion (1) and (2) follows from the definition. 

When $\mathscr{V} _1$ and $\mathscr{V} _2$ are finite dimensional, (3) follows by definition. Therefore, $f^A_{\real, \Hilb _\G}$ is also isomorphic by continuity of the functor $\lu{\phi}\K ^\G_{1,\tau,c}(\blank)_\vD$, that is, $\varinjlim \lu{\phi}\K ^\G_{1,\tau,c}(A_n)_\vD \cong \lu{\phi}\K ^\G_{1,\tau,c}(\varinjlim A_n)_\vD$. It can be checked in the same way as the case of $\K$-groups for $\Cst$-algebras by Lemma \ref{lem:sym} (see for example Theorem 6.3.2 of \cite{MR1783408}). Moreover, for any $\phi$-twisted representation $\mathscr{V}$ of $\G$, $f_{\mathscr{V},\Hilb _\G}^A=f_{\real , \Hilb _\G}^{A \hotimes \Kop (\mathscr{V})}$ since $\Hilb _\G \cong \mathscr{V} \otimes \Hilb _\G$. Finally, $f_{\mathscr{V}_1 , \mathscr{V}_2}^A$ is isomorphic for general $\mathscr{V}_1$ and $\mathscr{V}_2$ since the diagram
\[
\xymatrix{
\lu \phi \K ^\G_{1,c,\tau}(A \hotimes \Kop (\mathscr{V}_1))_\vD \ar[r]^{f_{\mathscr{V}_1,\mathscr{V}_2}^A} \ar[rd]_{f^A_{\mathscr{V}, \mathscr{V}_2 \oplus \Hilb _\G}} & \lu \phi \K ^\G_{1,c,\tau}(A \hotimes \Kop (\mathscr{V}_1 \oplus \mathscr{V}_2))_\vD \ar[d]^{f^A_{\mathscr{V}_1 \oplus \mathscr{V}_2, \Hilb _\G}} \\
&\lu \phi \K ^\G_{1,c,\tau}(A \hotimes \Kop (\mathscr{V}_1 \oplus \mathscr{V}_2 \oplus \Hilb _\G))_\vD 
}
\]
commutes, (\ref{form:Kas}) implies $\mathscr{V}_2 \oplus \Hilb _\G \cong \Hilb _\G$.

To see (4), it suffices to check that $\Ker \pi _*=\Im \iota _*$. We may assume that $A$ is unital. Let $\tilde{s}$ be a super $\G$-invariant odd symmetry in $A$ and let $s_t$ be a homotopy of super $\G$-invariant odd symmetries in $\pi (\tilde{s})=s_0$. Then there are $t_1,t_2,\dots,t_n$ such that $\ssbk{s_{t_i}-s_{t_j}}<2$. Therefore, it suffices to show that for $\tilde{s} \in \lu \phi F_{c,\mathscr{V}}^\G(\hat{\Mop}_{2,2}^c(A))$ and $s' \in \lu \phi F_{c,\mathscr{V}}^\G(\hat{\Mop}_{2,2}^c(A/I))$ such that $\ssbk{\pi (\tilde{s})-s'}<2$, there is a lift $\tilde{s}' \in \lu \phi F_{c,\mathscr{V}}^\G(\hat{\Mop}_{2,2}^c(A))$ of $s$ such that $\tilde{s}$ and $\tilde{s}'$ are homotopic. This follows from Lemma \ref{lem:sym} and the fact that every even $\G$-invariant self-adjoint $h \in \hat{\Mop}_{2,2}^c(A/I)\hotimes \Kop(\mathscr{V}) $ anticommuting with $s$ has an even $\G$-invariant self-adjoint lift $\tilde{h} \in \hat{\Mop}_{2,2}^c(A/I)\hotimes \Kop(\mathscr{V}) $ anticommuting with $\tilde{s}$. 
\end{proof}

Lemma \ref{lem:vDisK} asserts that the functor $\lu \phi \K _{1,c,\tau}^\G(\blank)_\vD$ is equivariantly stable, homotopy invariant and half-exact. Hence, the same argument as Proposition 4.1 of \cite{MR750677} can be applied for $\lu \phi \K _{1,c,\tau}^\G(\blank)_\vD$ and we obtain the long exact sequence 
\ma{
\dots \to \lu \phi \K _{1,c,\tau}^\G(SI) _\vD & \to \lu{\phi}\K_{1,c,\tau}^\G (SA)_\vD \to \lu{\phi}\K_{1,c,\tau}^\G (SA/I)_\vD \xra{\partial} \\
&\lu{\phi}\K _{1,c,\tau}^\G(I)_\vD \to \lu{\phi}\K_{1,c,\tau}^\G (A)_\vD \to \lu{\phi}\K_{1,c,\tau}^\G (A/I)_\vD.}
Here the boundary map $\partial$ is given in the same way as Proposition 4.7 of \cite{MR947500}.

Moreover, this exact sequence is extended for negative direction. Let $S^{p,q}$ denote the Real $\Cst$-algebra $C_0(\real ^{p+q})$ together with the Real structure $\overline{f}(\xi, \eta)=\overline{f(\xi,-\eta)}$.

\begin{prp}\label{prp:Bott}
There is a natural isomorphism between $\lu \phi \K _{1,c,\tau}^\G(A)_\vD$ and $\lu \phi \K _{1,c,\tau}^\G (S^{1,1}A)_\vD$.
\end{prp}
\begin{proof}
Let $\mathcal{T}_{0,\real}$ be the Real reduced Toeplitz algebra $C^*(s-1)$ where $s \in \Bop (\ell ^2 _\real \zahl _{\geq 0})$ is the unilateral shift operator $(s \xi)_i:=\xi _{i-1}$. Then, the boundary homomorphism $\partial : \K _{1,c,\tau}(S^{1,1}A)_\vD \to \K _{1,c,\tau}(A)_\vD$ associated to the exact sequence
\[
0 \to A \hotimes \Kop \to A \hotimes \mathcal{T}_{0,\real} \to S^{0,1}A \to 0
\]
gives a natural homomorphism from $\lu \phi \K _{1,c,\tau}^\G(S^{1,1}\blank)_\vD$ to $\lu \phi \K _{1,c,\tau}^\G (\blank)_\vD$. Moreover, applying Proposition 4.3 of \cite{MR750677} for the functor $\lu \phi \K _{1,c,\tau}^\G (A \hotimes \blank )_\vD$, we obtain $\lu \phi \K _{1,c,\tau}^\G(A \hotimes \mathcal{T}_{0,\real})_\vD=0$ and consequently $\partial$ is an isomorphism. 
\end{proof}

Let $\lu \phi \K _{1+p-q,c,\tau}^\G (A)_\vD$ denote the group $\lu \phi \K _{1,c,\tau}^\G (S^{p,q}A)_\vD$, which depends only on the difference $p-q$. Together with Proposition \ref{prp:Bott}, we obtain the following.

\begin{cor}
Let $0 \to I \to A \to A/I \to 0$ be an exact sequence of $\phi$-twisted $\G$-$\Cst$-algebras. Then, we obtain the long exact sequence 
\ma{
\dots \to &\lu{\phi}\K_{2,c,\tau}^\G (A/I)_\vD \to \lu{\phi}\K _{1,c,\tau}^\G(I)_\vD \to \lu{\phi}\K_{1,c,\tau}^\G (A)_\vD \to \lu{\phi}\K_{1,c,\tau}^\G (A/I)_\vD\\
 \to &\lu{\phi}\K _{0,c,\tau}^\G(I)_\vD \to \lu{\phi}\K_{0,c,\tau}^\G (A)_\vD \to \lu{\phi}\K_{0,c,\tau}^\G (A/I)_\vD \to \lu{\phi}\K _{-1,c,\tau}^\G(I)_\vD \to \dots.
}
\end{cor}

Finally we relate $\lu \phi \K _{1,c,\tau}^\G(\blank)_\vD$ with Definition \ref{def:twK}. For a $\phi$-twisted $\Zt$-graded $\G$-$\Cst$-algebra $A$, let $M_s(A)$ denote the stable multiplier algebra $M(A \hotimes  \Kop _\G)$ and let $Q_s(A)$ denote the stable corona algebra $M_s(A)/(A \hotimes \Kop _\G)$.

\begin{lem}\label{lem:Kuip}
Let $A$ be a $\phi$-twisted $\Zt$-graded $\G$-$\Cst$-algebra. Then, we have $\lu{\phi}\K_{*,c,\tau}^\G(M_s(A))_{\vD}=0$ for any $\ast \in \zahl$.
\end{lem}
\begin{proof}
If suffices to show that $[s]=0$ for any $s \in \lu \phi F_{c,\mathscr{V}}^\G((\hat{\Mop}_{2,2}^c((S^{p,q}M_s(A))^+))$. 
Let $V_n: \real \to \ell ^2 \zahl $ be isometries onto $\comp \cdot \delta _n$. We identify $\Hilb_{\G,A} \otimes \ell ^2\zahl _{>0}$ and $\Hilb_{\G,A} \otimes \ell ^2\zahl _{>0}$ with $\Hilb_{\G,A}$ by a unitary equivalence. Since the bilateral shift operator is homotopic to the identity in $\mathcal{U}\Bop (\ell ^2 \zahl)$, there is a homotopy
\[
\sum _{n<0} V_n s V_n^* \oplus s \oplus \sum_{n>0} V_n v_0 V_n^* \sim \sum _{n<0} V_n s V_n ^* \oplus v_0 \oplus \sum _{n >0} V_n v_0 V_n^*
\]
in $\lu \phi F _{c,V^3}^\G (\hat{\Mop} _{2,2}^c((S^{p,q}M_s(A))^+))$. This implies that
\[
[\sum _{n<0} V_n s V_n^*]+[s]+[\sum_{n>0} V_n v_0 V_n^*]=[\sum _{n<0} V_n s V_n^*]+[v_0]+[\sum_{n>0} V_n v_0 V_n^*]
\]
and hence $[s]=[v_0]=0$ in $\lu \phi \K _{1+p-q,c,\tau}^\G(M_s(A))$.
\end{proof}

\begin{thm}\label{thm:KvD}
There is a natural isomorphism $\lu \phi \K _{0,c,\tau}^\G(SA) \to \lu \phi \K_{1,c,\tau}^\G (A)_\vD $.
\end{thm}
\begin{proof}
This proof is based on the idea commented in Section 2 of \cite{MR2082096}. First we claim that $\lu \phi \KK _{c,\tau}^\G (\real , A)$ is isomorphic to the twisted equivariant $\K$-group $\lu \phi \K _{1,c,\tau}^\G(Q_s(A))_\vD$. Let $s$ be an element in $\lu \phi F_{c,V}^\G(Q_s(A))$. When we choose a self-adjoint lift $\tilde{s} \in \hat{\Mop}_{2,2}^c (M_s(A)) \hotimes \Kop (\mathscr{V})$ of $s$, then $(\hat {\real}^{2,2}_c \hotimes \Hilb _{\G,A} \hotimes \mathscr{V},1, \tilde{s})$ is a Kasparov $\real$-$A$ bimodule. The correspondence $[s] \mapsto [\hat {\real}^{2,2}_c \hotimes \Hilb _{\G,A} \hotimes \mathscr{V},1, \tilde{s}]$ determines a well-defined group homomorphism $\Phi_A: \lu \phi \K _{1,c,\tau}^\G(Q_sA)_\vD \to \lu \phi \KK ^\G _{c,\tau }(\real ,A)_{\mathrm{oh}}$.

Surjectivity of $\Phi_A$ follows from the Kasparov stabilization theorem (\ref{form:Kas}) and the fact that $\hat{\real}^{2,2}_c \hotimes \Hilb _{\G,A} \hotimes \mathscr{V} \cong \lu \phi \Hilb _{\G,A}^{c,\tau}$. 
Actually, every element in $\lu \phi \KK _{c,\tau}^\G (\real , A)$ is represented by a Kasparov bimodule of the form $[\lu \phi \Hilb_{\G,A}^{c,\tau},1,F]$. To see that $\Phi_A$ is injective, it suffices to show that $s \in \lu \phi F_{c,\mathscr{V}}^\G(A)$ satisfies $[s]=0$ if there is a degenerate Kasparov bimodule $[E,1,F]$ such that $\tilde{s} \oplus F$ is homotopic to a super $\G$-invariant odd symmetry in the space of Fredholm operators in $\Lop (\lu \phi \Hilb_{\G,A}^{c,\tau} \oplus E)$. By (\ref{form:Kas}), we may replace $E$ with $\Hilb _{\G,A}^{c,\tau}$. Hence the assumption means that $[s]+[\pi (F)]$ is in the image of $\pi_*:\lu \phi \K _{1,c,\tau}^\G (M_s(A))_\vD \to \lu \phi \K _{1,c,\tau}^\G (Q_s(A))_\vD$ and we obtain $[s]=0$ by Lemma \ref{lem:Kuip}.

Consider the map between long exact sequences 
\[
\xymatrix{
0 \ar[r] & SA \hotimes \Kop _\G \ar[r] \ar@{=}[d] & SM_s(A) \ar[d] \ar[r] & SQ_s(A) \ar[r]\ar[d]  &0\\0 \ar[r] & SA \hotimes \Kop _\G \ar[r] & M_s(SA) \ar[r] & Q_s(SA) \ar[r] &0
}
\]
and apply Lemma \ref{lem:Kuip}. Then we obtain that the canonical inclusion $SQ_s(A) \to Q_s(SA)$ induces an isomorphism $\lu \phi \K_{1,c,\tau}^\G(SQ_s(A))_\vD \cong \lu \phi \K_{1,c,\tau}^\G(Q_s(SA))_\vD$. Finally we obtain the isomorphism
\[
\Psi_A: \lu \phi \KK _{c,\tau}^\G (\real , SA) \xra{S\Phi _A^{-1}} \lu \phi \K _{1,c,\tau}^\G (SQ_sA)_\vD \xra{\partial} \lu \phi \K_{1,c,\tau}^\G (A \hotimes \Kop _\G)_\vD
\]
such that all of these homomorphisms are natural. 
\end{proof}

As a consequence of Theorem \ref{thm:KvD}, we obtain the isomorphism 
\[
\beta _A:=\Psi_A \circ \beta \circ \Psi_A ^{-1}  : \lu \phi \K _{1,c,\tau}^\G(A)_\vD \to \lu \phi \K _{1,c,\tau}^\G(SA \hotimes \Cl _{0,1})_\vD .
\]
where $\beta :=[C_0(\real) \hotimes \Cl _{0,1},1,x(1+x^2)^{-1/2}e] \in \KK ^\G(\real , S\Cl _{0,1})$ is the Bott element. On the other hand, a natural isomorphism $\Xi _A$ from $\lu \phi \K _{1,c,\tau}^\G(A)_\vD$ to $\lu \phi \K _{1,c,\tau}^\G(SA \hotimes \Cl _{0,1})_\vD$ is given in Theorem 2.14 of \cite{MR961241} as $\Xi _A[s]:= [z(s;t)]$ where 
\ma{z(s;t)&:= \Ad ((1+ve)/ \sqrt{2}) \circ \Ad \nu (s;t) \circ \Ad \nu (v;-t) (e),\\
\nu (x;t)&:= \cos (\pi t /2) +xe \sin (\pi t /2).}
Actually, these isomorphisms coincide. To see this, let $s$ be an odd $c$-graded $\G$-invariant symmetry in $Q_s(A) \hotimes \Kop (\mathscr{V})$ and let $\tilde{s}$ be a self-adjoint lift of $s$ in $\hat \Mop _{2,2}^c M_s(A) \hotimes \Kop (\mathscr{V})$. Then,  
\ma{[ \Hilb _{\G,SA \hotimes \Cl _{0,1}} ^{c,\tau}, 1, z(\tilde{s};t)] &=[\Hilb _{\G,SA \hotimes \Cl _{0,1}} ^{c,\tau}, 1, \Ad \nu (\tilde{s};t) \circ \Ad \nu (v;-t) (e)] \\
&=[ \Hilb _{\G,SA \hotimes \Cl _{0,1}} ^{c,\tau}, 1, \Ad \nu (\tilde{s};t) (e)]\\
&=[ \Hilb _{\G,SA \hotimes \Cl _{0,1}} ^{c,\tau}, 1, e \cos (\pi t) + \tilde{s} \sin (\pi t) ]}
and the operator $e \cos (\pi t) + \tilde{s} \sin (\pi t)$ satisfies (\ref{form:KK}) with respect to the Kasparov product $[ \Hilb _{\G,A} ^{c,\tau}, 1, \tilde{s}] \otimes _\real \beta$. Consequently we obtain $\Phi _{SA \hotimes \Cl _{0,1}}  \Xi_{Q_s(A)} [s]= (\beta \otimes _\real (\Phi _A [s]))$, which implies $\Xi =\beta _A$ by naturality of $\Xi$.

Moreover, there is another canonical identification of these groups. When we write a $c$-graded $\G$-invariant odd symmetry in $(\hat \Mop _{2,2}^c(A) \hotimes \Kop (\mathscr{V})) \hotimes \Cl _{0,1}$ as $s_0+s_1e$ where $e$ is the odd self-adjoint generator of $\Cl _{0,1}$, then $s_0$ and $s_1$ are self-adjoint element such that $s_0$ is odd, $s_1$ is even and satisfy $s_0s_1=s_1s_0$ and $s_0^2+s_1^2=1$. Therefore, by Proposition \ref{prp:twK}, the correspondence $(s_1 +is_0) \mapsto (s_0+es_1)$ induces the natural isomorphism $\Theta_A$ from $\lu \phi \K _{0,c,\tau}^\G (A)$ to $\lu \phi \bK _{c,\tau}^\G (A \hotimes \Cl _{0,1})_e$ (which turns out to be isomorphic to $\lu \phi \K _{1,c,\tau}^\G(A \hotimes \Cl_{0,1})$ by Proposition \ref{prp:indep}). 

\begin{lem}\label{lem:compati}
The isomorphisms $\Psi_{A \hotimes \Cl _{0,1}} \circ \beta$ and $\Theta _A$ coincide.
\end{lem}
\begin{proof}
By the above argument, we have $\Psi_{A \hotimes \Cl _{0,1}} \circ \beta= \partial \circ \beta _A \circ \Phi _A =\partial \circ \Xi _A $. For a $c$-graded $\G$-invariant odd Fredholm operator $F \in M_s(A) \hotimes \Kop (\mathscr{V})$, the composition $\partial \circ \Xi _A$ is calculated in Proposition 3.4 of \cite{MR961241} as
\[
\partial \circ \Xi_A[F]=[-\Ad (\frac{1}{\sqrt{2}}(1+ve))(\sin \pi F + e \cos \pi F)].
\]
On the other hand, $\Theta _A [-\exp (\pi i F)]= [-\Ad (\frac{1}{\sqrt{2}}(1+ve))(\cos \pi F + i\sin \pi F)]$ by definition and Remark \ref{rmk:indep}.
\end{proof}

\begin{prp}\label{prp:bdexp}
There is an exact sequence 
\[
\lu{\phi}\K _{1,c,\tau}^\G(A)_\vD \to \lu{\phi}\K _{1,c,\tau}^\G(A/I)_\vD \xra{\partial} \lu{\phi}\K _{0,c,\tau}^\G(I) \to \lu{\phi}\K_{0,c,\tau}^\G (A)
\]
were the boundary map given by $\partial [s]=-\exp (\pi i \tilde{s})$ where $\tilde{s}$ is a self-adjoint element in $A$ such that $\pi (\tilde{s})=s$.
\end{prp}
\begin{proof}
Consider the following diagram
\[
\xymatrix@C=1em{\scriptstyle \lu{\phi}\K _{1,c,\tau}^\G(A)_\vD \ar[r] \ar[d]^{\beta _A} &\scriptstyle \lu{\phi}\K _{1,c,\tau}^\G(A/I)_\vD \ar@{.>}[r] \ar[d]^{\beta _{A/I}} &\scriptstyle \lu{\phi}\K _{0,c,\tau}^\G(I) \ar[r] &\scriptstyle \lu{\phi}\K_{1,c,\tau}^\G (A)
\\
\scriptstyle \lu{\phi}\K _{1,c,\tau}^\G(SA \hotimes \Cl _{0,1})_\vD \ar[r]&\scriptstyle \lu{\phi}\K _{1,c,\tau}^\G(SA/I \hotimes \Cl _{0,1})_\vD \ar[r]^\partial &\scriptstyle \lu{\phi}\K _{1,c,\tau}^\G(I \hotimes \Cl _{0,1})_\vD \ar[r] \ar[u]_{\Theta _I^{-1}} &\scriptstyle \lu{\phi}\K_{1,c,\tau}^\G (A \hotimes \Cl_{0,1})_\vD \ar[u]_{\Theta _I^{-1}}
}
\]
where the second row is the long exact sequence associated with the exact sequence $0 \to I\hotimes \Cl_{0,1} \to A \hotimes \Cl_{0,1} \to A/I \hotimes \Cl_{0,1} \to 0$. 

Since $\Theta _I^{-1} \circ \partial \circ \beta _{Q_s(I)} =\beta ^{-1} \circ S\Phi _{I \hotimes \Cl _{0,1}} \circ \beta _{Q_s(I)}=\Phi _A$, we obtain $\Theta _I^{-1} \circ \partial \circ \beta _{A/I} =\Phi _I \circ \mu_*$ since the diagram
\[
\xymatrix{
\lu \phi \K _{1,c,\tau}^\G (A/I)_\vD  \ar[d]^{\mu_*} \ar[r] & \lu \phi \K _{1,c,\tau}^\G (I) \ar@{=}[d] \\
\lu \phi \K _{1,c,\tau}^\G (Q_s(I))_\vD \ar[r] & \lu \phi \K _{1,c,\tau}^\G (I)
}
\]
commutes.
For an odd self-adjoint $c$-graded $\G$-invariant symmetry $s \in \hat{\Mop} _{2,2}(A/I) \hotimes \Kop(\mathscr{V})$, we have $\Phi_{I} \circ \mu_*([s])=[-\exp(\pi i \mu(\tilde{s}))]$ under the identification in Proposition \ref{prp:twK} where $\tilde{s}$ is an odd self-adjoint $c$-graded $\G$-invariant lift of $s$ in $\hat{\Mop} _{2,2}A \hotimes \Kop(\mathscr{V})$.
\end{proof}

Let $A$ be a trivially graded unital $\phi$-twisted $\G$-$\Cst$-algebra. 
We say that an element $x$ in $A$ is \emph{$c$-twisted $\G$-invariant} if $\alpha _g(x)=(-1)^{c(g)}(x)$ for any $g \in \G$. 
For a finite dimensional $(\phi,c,\tau)$-twisted unitary representation $\mathscr{V}$, let $\lu \phi \F _{c,\mathscr{V}}^\G (A)$ denote the space of $c$-twisted $\G$-invariant symmetries in $A  \otimes \Kop({\mathscr{V}}^\circ)$. 
Here we write $\mathscr{V}^\circ$ for the unitary representation of $\G$ obtained by forgetting the $\Zt$-grading of $\mathscr{V}$. 
In the same way as the definition of $\lu \phi \bK_{c,\tau}^\G(A)_e$, the inductive limit $\lu \phi \mathscr{K}_{0,c,\tau}^\G (A):= \varinjlim  _\mathscr{V} \pi _0 \lu \phi \F _{c,\mathscr{V}}^\G (A)$ with respect to the inclusion $\lu \phi \F _{c,\mathscr{V}}^\G (A) \subset \lu \phi \F _{c, \mathscr{V}\oplus \mathscr{W}}^\G (A)$ given by $s \mapsto s \oplus \gamma_\mathscr{W}$ (here $\gamma _\mathscr{W}$ is the grading operator of $\mathscr{W}$) has a structure of abelian group. 
Similarly, let $\lu \phi \mathscr{U} _{c, \mathscr{V} }^\G (A)$ denote the space of unitaries in $A \otimes \Kop (\mathscr{V}^\circ)$ such that $\alpha _g(u)=u$ if $c(g)+\phi (g)=0$ and $\alpha _g(u)=u^*$ if $c(g)+\phi(g)=1$. Set $\lu \phi \mathscr{K}_{-1,c,\tau}^\G (A):=\varinjlim _\mathscr{V} \lu \phi \mathscr{U}_{c,\mathscr{V} }^\G (A)$ with respect to the inclusion $\lu \phi \mathscr{U}_{c,\mathscr{V}}^\G (A) \subset \lu \phi \mathscr{U} _{c,\mathscr{V}\oplus \mathscr{W}}^\G (A)$ given by $u \mapsto u \oplus 1_{\mathscr{W}}$. 

For a nonunital $\phi$-twisted $\G$-$\Cst$-algebra $A$, the $\lu \phi \mathscr{K}_{*,c,\tau}^\G (A)$ is defined by the kernel of $\lu \phi \mathscr{K}_{*,c,\tau}^\G (A^+) \to \lu \phi \mathscr{K}_{*,c,\tau}^\G (C_b(\G ^0))$.
\begin{thm}\label{cor:triv}
Let $A$ be a trivially graded $\phi$-twisted $\G$-$\Cst$-algebra. 
\begin{enumerate}
\item The group $\lu \phi \mathscr{K}_{0,c,\tau}^\G (A)$ is isomorphic to $\lu{\phi}\K_{0,c,\tau}^\G(A)_\vD$.
\item The group $\lu \phi \mathscr{K}_{-1,c,\tau}^\G (A)$ is isomorphic to $\lu{\phi}\K_{-1,c,\tau}^\G(A)_\vD$. 
\item Let $0 \to I \to A \to A/I \to 0$ be an exact sequence of $\phi$-twisted trivially graded $\Cst$-algebras. Then there is an exact sequence 
\[
\lu{\phi}\mathscr{K}_{0,c,\tau}^\G(A) \to \lu{\phi}\mathscr{K}_{0,c,\tau}^\G(A/I) \xra{\partial} \lu{\phi}\mathscr{K} _{-1,c,\tau}^\G(I) \to \lu{\phi}\mathscr{K}_{-1,c,\tau}^\G (A)
\]
where the boundary map $\partial$ is given by $\partial [s] = [-\exp (\pi i \tilde{s})]$ where $\tilde{s}$ is a self-adjoint lift of $s$.
\end{enumerate}
\end{thm}
\begin{proof}
We may assume $A$ is unital. The isomorphism $\lu \phi \bK_{c,\tau}^\G(A \hotimes \Cl _{0,1})_e \cong \lu \phi \mathscr{K}_{0,c,\tau}^\G (A)$ is induced from the bijection $\lu \phi F_{c,\mathscr{V}}^\G (A \hotimes \Cl _{0,1}) \to \lu \phi \mathscr{F} _{c,\mathscr{V}}^\G (A)$ given by
\[
s=s_0e +s_1 \mapsto s'=s_0 \gamma _{\mathscr{V}} +s_1.
\]
Here, the decomposition $s'=s_0\gamma +s_1$ is given by the $\pm 1$ component of the involution $\Ad \gamma _{\mathscr{V}}$. This map determines a well-defined bijection because $s_0e+s_1$ is in $\lu \phi F_{c,\mathscr{V}}^\G (A \hotimes \Cl _{0,1})$ if and only if $s_0$ is even, $s_1$ is odd, $s_i^*=s_i$, $s_0s_1=s_1s_0$, $s_0^2+s_1^2=1$, $\alpha _g (s_0)=s_0$ and $\alpha _g(s_1)=(-1)^{c(g)}s_1$. 

The isomorphism (2) is also given in the same way as above. For a $\Zt$-graded $\phi$-twisted unital $\G$-$\Cst$-algebra $A$, let $\lu \phi U_{c,\mathscr{V}}^\G (A)$ be the space of unitaries in $A \hotimes \Kop (\mathscr{V})$ such that $\gamma (u)=u^*$ and $\alpha _g(u)=\gamma ^{\phi (g)+c(g)}(u)$. Then, by Proposition \ref{prp:twK} and Lemma \ref{lem:compati}, $\lu \phi \K _{0,c,\tau}^\G (A) \cong \varinjlim _{\mathscr{V}} \pi _0 \lu \phi U_{c,\mathscr{V}}^\G (A)$. Now, the isomorphism $\lu \phi \K_{0,c,\tau}^\G(A \hotimes \Cl _{0,1}) \cong \lu \phi \mathscr{K}_{-1,c,\tau}^\G (A)$ is induced from the bijection from $\lu \phi U_{c,\mathscr{V}}^\G (A \hotimes \Cl _{0,1})$ to $\lu \phi \mathscr{U} _{c,\mathscr{V}}^\G (A)$ given by
\[
u=u_0e +u_1 \mapsto u'=u_0\gamma +u_1 .
\]
Here, the decomposition $u'=u_0\gamma +u_1$ is given by the $\pm 1$ component of the involution $x \mapsto \gamma_{\mathscr{V}} x^*\gamma _{\mathscr{V}}$. 

Finally, (3) follows from Proposition \ref{prp:bdexp}. Note that the identification $\lu \phi \bK_{1,c,\tau}^\G(A \hotimes \Cl _{0,1})_e \cong \lu \phi \K_{1,c,\tau}^\G(A\hotimes \Cl _{0,1})$ is given by $\Ad (1+ve)/\sqrt{2}$, which also appears in the isomorphism $\lu \phi \K _{0,c,\tau}^\G (A) \cong \varinjlim _{\mathscr{V}} \pi _0 \lu \phi U_{c,\mathscr{V}}^\G (A)$.
\end{proof}

These groups are related to Karoubi's $\K$-theory. In \citelist{\cite{MR0238927}\cite{MR2458205}}, Karoubi gives a definition of higher real $\K$-theory for pseudo-abelian Banach categories in terms of modules over Clifford algebras (Definition 4.11 of \cite{MR2458205}), which is applied by Thiang~\cite{mathph14067366} for the study of topological insulators. Now we generalize it for twisted equivariant $\K$-theory for trivially graded $\Cst$-algebras. A $(\phi ,\tau )$-twisted $(n,m)$-Karoubi triple is a triplet $(W,\Gamma _1 ,\Gamma _2)$ where $W$ is a finitely generated projective (i.e.\ $W \cong pA^n$ for some $p \in \Mop_n(A)$) $(\phi,\tau)$-twisted $\G$-equivariant Hilbert $A$-module with a unital $\phi$-twisted $\Zt$-graded $\G$-equivariant $\ast$-homomorphism $\varphi : \Cl _{n,m} \to \Lop (W)$ and $\Gamma _i$ are symmetries on $W$ which determine a $\Zt$-grading on $W$ compatible with $\varphi$ and the $\G$-action. A Karoubi triple $(W,\Gamma _1,\Gamma _2)$ is trivial if $\Gamma _1=\Gamma _2$. Let $\lu \phi \K _{(n,m),\tau}^\G(A)_{\Kar}$ denote the quotient of the semigroup of homotopy classes of Karoubi triples (product is given by the direct sum) by the subsemigroup of trivial triples. 

\begin{cor}\label{cor:real}
Let $A$ be a Real $\G$-$\Cst$-algebra. Then, the group $\lu \phi \K_{(n,m),\tau}^\G (A)_{\Kar }$ is isomorphic to $\lu \phi \K_{m-n, \tau}^\G(A)$. In particular, for a real twist $(\phi _\real,0, \tau _\real)$ of $\G_\real$, $\lu{\phi _\real} \K_{(n,m),\tau _\real }^{ \G_\real } (A)_{\Kar }$ is isomorphic to $\KR _{m-n,\tau}^\G(A)$.
\end{cor}
\begin{proof}
Let $\G'=\G _{n,m}$ and $(\phi,c',\tau')=(\phi, c_{n,m}, \tau_{n,m})$ be as in Example \ref{exmp:Cl2}. Now, we get a homomorphism
\[F_A: \lu {\phi} \mathscr{K}_{0,c',\tau'}^{\G'}(A) \to \lu \phi \K_{(n,m),\tau}^{\G} (A)_{\Kar }\]
mapping $s \in \lu \phi \mathscr{F}_{c',\mathscr{V}}^{\G'}(A)$ to $[\mathscr{V} \hotimes A, s, \gamma ]$. 

For a $(n,m)$-Karoubi triple $[W,\Gamma _1,\Gamma _2]$, we write $W_2$ for the $(\phi,c',\tau')$-twisted $\G'$-equivariant Hilbert $A$-module $W$ with the $\Zt$-grading $\Gamma _2$. It is embedded in $\lu \phi \Hilb _{\G',A}^{c',\tau'}$ by (\ref{form:Kas}). Since $W$ is finitely generated projective, the projection $P$ onto $W_2$ is compact by Lemma 6.3 of \cite{MR582160} and hence $p_n P p_n \to P$. Consequently, $W_2$ is embedded into $p_n(\lu \phi \Hilb _{\G',A}^{c',\tau'})=A \hotimes \mathscr{V}$, where $\mathscr{V}:=p_n(\lu \phi \Hilb _{\G'}^{c',\tau'})$.

It immediately implies that $F_A$ is surjective. Moreover, we can check injectivity by applying this argument for degenerate Karoubi triples and Hilbert $A \otimes C[0,1]$-modules given by homotopies of Karoubi triples. 

Now, the conclusion follows from Corollary \ref{cor:Cl} and Theorem \ref{cor:triv}.
\end{proof}

The following corollary immediately follows from Theorem Corollary \ref{cor:CT} and Theorem \ref{cor:triv}.

\begin{cor}
Let $(\G,\phi,c,\tau)$ be a CT-type symmetry as in Example \ref{exmp:CT} whose CT-type corresponds to $\KF_n$ as indicated in Table \ref{table:CT}. Let $A$ be a (trivially graded) Real $\G_0$-$\Cst$-algebra. Then;
\begin{enumerate}
\item The group $\lu \phi \mathscr{K} _{0,c,\tau}^\G (A)$ is isomorphic to $\KF_n ^{\G_0}(A)$.
\item The group $\lu \phi \mathscr{K}_{-1,c,\tau}^\G (A)$ is isomorphic to $\KF_{n-1}^{\G _0}(A)$. 
\item Let $0 \to I \to A \to A/I \to 0$ be an exact sequence of $\phi$-twisted trivially graded $\Cst$-algebras. Then there is an exact sequence 
\[
\lu{\phi}\mathscr{K}_{0,c,\tau}^\G(A) \to \lu{\phi}\mathscr{K}_{0,c,\tau}^\G(A/I) \xra{\partial} \lu{\phi}\mathscr{K} _{-1, c,\tau}^\G(I) \to \lu{\phi}\mathscr{K}_{-1,c,\tau}^\G (A)
\]
where the boundary map $\partial$ is given by $\partial [s] = [-\exp (\pi i \tilde{s})]$ where $\tilde{s}$ is a self-adjoint lift of $s$.
\end{enumerate}
\end{cor}

\begin{exmp}
Let $A$ be a Real $\Cst$-algebra and let
\[
0 \to A \otimes \Kop  \to A \otimes \mathcal{T}_\real \to A \otimes C(\mathbb{T}^{0,1}) \to 0
\]
be the Toeplitz extension associated to $A$. Then, the boundary homomorphism $\partial =-\exp (\pi i \blank)$ from $\lu \phi \mathscr{K} _{0,c,\tau}^\sA (A \otimes C(\mathbb{T}^{0,1})) \cong \KF _n(A) \oplus \KF_{n-1}(A)$ to $\lu \phi \mathscr{K} _{-1,c,\tau}^\sA (A) \cong \KF _{n-1}(A)$ is given by $0 \oplus \id_{\KF_{n-1}(A)}$.
 
In the theory of topological insulators and topological superconductors, gapped Hamiltonians of $d$-dimensional quantum systems with the CT symmetry $(\sA, \tau)$ are classified by the group $\lu \phi \mathscr{K} _{0,c,\tau}^{\sA}(C(\T ^{0,d}))$ and the corresponding edge Hamiltonians are classified by $\lu \phi \mathscr{K} _{-1,c,\tau}^{\sA}(C(\T ^{0,d-1}))$. The homomorphism $\partial$ gives a mathematical proof of the bulk-edge correspondence (cf.\ \cite{mathKT150907210}). The use of twisted equivariant $\K$-theory for the study of topological phases is discussed in \cite{Kubota2}.
\end{exmp}

\subsection*{Acknowledgment}
The author would like to thank his supervisor Yasuyuki Kawahigashi for his support and encouragement. He also would like to thank Yuki Arano, Mikio Furuta, Shin Hayashi, Motoko Kotani, Shinichiro Matsuo, Koji Sato and Guo Chuan Thiang for their helpful conversations. This work was supported by Research Fellow of the JSPS (No.\ 26-7081) and the Program for Leading Graduate Schools, MEXT, Japan.

\bibliographystyle{alpha}
\bibliography{bibABC,bibDEFG,bibHIJK,bibLMN,bibOPQR,bibSTUV,bibWXYZ,arxiv,bulkedge}

\def\cprime{$'$} \def\cprime{$'$} \def\cdprime{$''$} \def\cprime{$'$}
  \def\cprime{$'$} \def\ocirc#1{\ifmmode\setbox0=\hbox{$#1$}\dimen0=\ht0
  \advance\dimen0 by1pt\rlap{\hbox to\wd0{\hss\raise\dimen0
  \hbox{\hskip.2em$\scriptscriptstyle\circ$}\hss}}#1\else {\accent"17 #1}\fi}
  \def\cprime{$'$} \def\cdprime{$''$} \def\cprime{$'$} \def\cprime{$'$}
  \def\cprime{$'$} \def\cprime{$'$} \def\cprime{$'$} \def\cprime{$'$}
  \def\cprime{$'$} \def\cprime{$'$} \def\cprime{$'$}
  \def\Dbar{\leavevmode\lower.6ex\hbox to 0pt{\hskip-.23ex \accent"16\hss}D}
  \def\cftil#1{\ifmmode\setbox7\hbox{$\accent"5E#1$}\else
  \setbox7\hbox{\accent"5E#1}\penalty 10000\relax\fi\raise 1\ht7
  \hbox{\lower1.15ex\hbox to 1\wd7{\hss\accent"7E\hss}}\penalty 10000
  \hskip-1\wd7\penalty 10000\box7}
  \def\cfudot#1{\ifmmode\setbox7\hbox{$\accent"5E#1$}\else
  \setbox7\hbox{\accent"5E#1}\penalty 10000\relax\fi\raise 1\ht7
  \hbox{\raise.1ex\hbox to 1\wd7{\hss.\hss}}\penalty 10000 \hskip-1\wd7\penalty
  10000\box7} \def\cprime{$'$} \def\cprime{$'$} \def\cdprime{$''$}
  \def\polhk#1{\setbox0=\hbox{#1}{\ooalign{\hidewidth
  \lower1.5ex\hbox{`}\hidewidth\crcr\unhbox0}}} \def\cprime{$'$}
  \def\cprime{$'$} \def\cprime{$'$}
  \def\cftil#1{\ifmmode\setbox7\hbox{$\accent"5E#1$}\else
  \setbox7\hbox{\accent"5E#1}\penalty 10000\relax\fi\raise 1\ht7
  \hbox{\lower1.15ex\hbox to 1\wd7{\hss\accent"7E\hss}}\penalty 10000
  \hskip-1\wd7\penalty 10000\box7} \def\cprime{$'$}
  \def\cydot{\leavevmode\raise.4ex\hbox{.}}
  \def\cydot{\leavevmode\raise.4ex\hbox{.}}
  \def\cydot{\leavevmode\raise.4ex\hbox{.}}
  \def\cydot{\leavevmode\raise.4ex\hbox{.}}
  \def\cydot{\leavevmode\raise.4ex\hbox{.}}
  \def\cydot{\leavevmode\raise.4ex\hbox{.}} \def\cdprime{$''$}
  \def\cdprime{$''$} \def\cdprime{$''$} \def\cdprime{$''$}
  \def\polhk#1{\setbox0=\hbox{#1}{\ooalign{\hidewidth
  \lower1.5ex\hbox{`}\hidewidth\crcr\unhbox0}}} \def\cprime{$'$}
  \def\cprime{$'$} \def\cprime{$'$} \def\cprime{$'$} \def\cprime{$'$}
  \def\cprime{$'$} \def\cprime{$'$} \def\cprime{$'$} \def\cprime{$'$}
  \def\cprime{$'$} \def\polhk#1{\setbox0=\hbox{#1}{\ooalign{\hidewidth
  \lower1.5ex\hbox{`}\hidewidth\crcr\unhbox0}}}
  \def\polhk#1{\setbox0=\hbox{#1}{\ooalign{\hidewidth
  \lower1.5ex\hbox{`}\hidewidth\crcr\unhbox0}}} \def\cprime{$'$}
  \def\cprime{$'$} \def\cprime{$'$}
  \def\polhk#1{\setbox0=\hbox{#1}{\ooalign{\hidewidth
  \lower1.5ex\hbox{`}\hidewidth\crcr\unhbox0}}}
  \def\polhk#1{\setbox0=\hbox{#1}{\ooalign{\hidewidth
  \lower1.5ex\hbox{`}\hidewidth\crcr\unhbox0}}}
  \def\polhk#1{\setbox0=\hbox{#1}{\ooalign{\hidewidth
  \lower1.5ex\hbox{`}\hidewidth\crcr\unhbox0}}}
  \def\polhk#1{\setbox0=\hbox{#1}{\ooalign{\hidewidth
  \lower1.5ex\hbox{`}\hidewidth\crcr\unhbox0}}} \def\cprime{$'$}
  \def\cprime{$'$} \def\cprime{$'$} \def\cprime{$'$} \def\cprime{$'$}
  \def\cprime{$'$} \def\cprime{$'$}
\begin{bibdiv}
\begin{biblist}

\bib{MR2172633}{article}{
      author={Atiyah, Michael},
      author={Segal, Graeme},
       title={Twisted {$K$}-theory},
        date={2004},
        ISSN={1810-3200},
     journal={Ukr. Mat. Visn.},
      volume={1},
      number={3},
       pages={287\ndash 330},
      review={\MR{2172633 (2006m:55017)}},
}

\bib{MR1043170}{book}{
      author={Atiyah, M.~F.},
       title={{$K$}-theory},
     edition={Second},
      series={Advanced Book Classics},
   publisher={Addison-Wesley Publishing Company, Advanced Book Program, Redwood
  City, CA},
        date={1989},
        ISBN={0-201-09394-4},
        note={Notes by D. W. Anderson},
      review={\MR{1043170 (90m:18011)}},
}

\bib{mathKT150907210}{misc}{
      author={Bourne, Chris},
      author={Carey, Alan~L.},
      author={Rennie, Adam},
       title={{A} noncommutative framework for topological insulators},
         how={preprint},
        date={2015},
        note={preprint, \href{http://arxiv.org/abs/1509.07210}{{\tt
  arXiv:1509.07210[math-ph]}}},
}

\bib{MR0463928}{article}{
      author={Brown, Lawrence~G.},
      author={Green, Philip},
      author={Rieffel, Marc~A.},
       title={Stable isomorphism and strong {M}orita equivalence of
  {$C\sp*$}-algebras},
        date={1977},
        ISSN={0030-8730},
     journal={Pacific J. Math.},
      volume={71},
      number={2},
       pages={349\ndash 363},
      review={\MR{0463928 (57 \#3866)}},
}

\bib{MR1656031}{book}{
      author={Blackadar, Bruce},
       title={{$K$}-theory for operator algebras},
     edition={Second},
      series={Mathematical Sciences Research Institute Publications},
   publisher={Cambridge University Press, Cambridge},
        date={1998},
      volume={5},
        ISBN={0-521-63532-2},
      review={\MR{1656031 (99g:46104)}},
}

\bib{MR1857079}{article}{
      author={Chabert, J{\'e}r{\^o}me},
      author={Echterhoff, Siegfried},
       title={Twisted equivariant {$KK$}-theory and the {B}aum-{C}onnes
  conjecture for group extensions},
        date={2001},
        ISSN={0920-3036},
     journal={$K$-Theory},
      volume={23},
      number={2},
       pages={157\ndash 200},
         url={http://dx.doi.org/10.1023/A:1017916521415},
      review={\MR{1857079 (2002m:19003)}},
}

\bib{MR750677}{incollection}{
      author={Cuntz, Joachim},
       title={{$K$}-theory and {$C^{\ast} $}-algebras},
        date={1984},
   booktitle={Algebraic {$K$}-theory, number theory, geometry and analysis
  ({B}ielefeld, 1982)},
      series={Lecture Notes in Math.},
      volume={1046},
   publisher={Springer, Berlin},
       pages={55\ndash 79},
         url={http://dx.doi.org/10.1007/BFb0072018},
      review={\MR{750677 (86d:46071)}},
}

\bib{MR0458185}{book}{
      author={Dixmier, Jacques},
       title={{$C\sp*$}-algebras},
   publisher={North-Holland Publishing Co., Amsterdam-New York-Oxford},
        date={1977},
        ISBN={0-7204-0762-1},
        note={Translated from the French by Francis Jellett, North-Holland
  Mathematical Library, Vol. 15},
      review={\MR{0458185 (56 \#16388)}},
}

\bib{MR0282363}{article}{
      author={Donovan, P.},
      author={Karoubi, M.},
       title={Graded {B}rauer groups and {$K$}-theory with local coefficients},
        date={1970},
        ISSN={0073-8301},
     journal={Inst. Hautes \'Etudes Sci. Publ. Math.},
      number={38},
       pages={5\ndash 25},
      review={\MR{0282363 (43 \#8075)}},
}

\bib{MR2499440}{article}{
      author={Emerson, Heath},
      author={Meyer, Ralf},
       title={Equivariant representable {K}-theory},
        date={2009},
        ISSN={1753-8416},
     journal={J. Topol.},
      volume={2},
      number={1},
       pages={123\ndash 156},
         url={http://dx.doi.org/10.1112/jtopol/jtp003},
      review={\MR{2499440 (2010a:19003)}},
}

\bib{MR936628}{book}{
      author={Fell, J. M.~G.},
      author={Doran, R.~S.},
       title={Representations of {$^*$}-algebras, locally compact groups, and
  {B}anach {$^*$}-algebraic bundles. {V}ol. 1},
      series={Pure and Applied Mathematics},
   publisher={Academic Press, Inc., Boston, MA},
        date={1988},
      volume={125},
        ISBN={0-12-252721-6},
        note={Basic representation theory of groups and algebras},
      review={\MR{936628 (90c:46001)}},
}

\bib{mathOA150500364}{misc}{
      author={Farsi, Carla},
      author={Gillaspy, Elizabeth},
       title={{T}wists over {R}enault-{D}eaconu groupoids and twisted vector
  bundles},
        date={2015},
        note={preprint, \href{http://arxiv.org/abs/1505.00364}{{\tt
  arXiv:1505.00364[math.OA]}}},
}

\bib{MR2860342}{article}{
      author={Freed, Daniel~S.},
      author={Hopkins, Michael~J.},
      author={Teleman, Constantin},
       title={Loop groups and twisted {$K$}-theory {I}},
        date={2011},
        ISSN={1753-8416},
     journal={J. Topol.},
      volume={4},
      number={4},
       pages={737\ndash 798},
         url={http://dx.doi.org/10.1112/jtopol/jtr019},
      review={\MR{2860342}},
}

\bib{MR3119923}{article}{
      author={Freed, Daniel~S.},
      author={Moore, Gregory~W.},
       title={Twisted equivariant matter},
        date={2013},
        ISSN={1424-0637},
     journal={Ann. Henri Poincar\'e},
      volume={14},
      number={8},
       pages={1927\ndash 2023},
         url={http://dx.doi.org/10.1007/s00023-013-0236-x},
      review={\MR{3119923}},
}

\bib{MR1608798}{incollection}{
      author={Grothendieck, Alexander},
       title={Le groupe de {B}rauer. {I}. {A}lg\`ebres d'{A}zumaya et
  interpr\'etations diverses},
        date={1995},
   booktitle={S\'eminaire {B}ourbaki, {V}ol.\ 9},
   publisher={Soc. Math. France, Paris},
       pages={Exp.\ No.\ 290, 199\ndash 219},
      review={\MR{1608798}},
}

\bib{MR2058474}{incollection}{
      author={Higson, Nigel},
      author={Guentner, Erik},
       title={Group {$C^\ast$}-algebras and {$K$}-theory},
        date={2004},
   booktitle={Noncommutative geometry},
      series={Lecture Notes in Math.},
      volume={1831},
   publisher={Springer, Berlin},
       pages={137\ndash 251},
         url={http://dx.doi.org/10.1007/978-3-540-39702-1_3},
      review={\MR{2058474 (2005c:46103)}},
}

\bib{MR1487204}{article}{
      author={Higson, Nigel},
      author={Kasparov, Gennadi},
       title={Operator {$K$}-theory for groups which act properly and
  isometrically on {H}ilbert space},
        date={1997},
        ISSN={1079-6762},
     journal={Electron. Res. Announc. Amer. Math. Soc.},
      volume={3},
       pages={131\ndash 142 (electronic)},
         url={http://dx.doi.org/10.1090/S1079-6762-97-00038-3},
      review={\MR{1487204 (99e:46090)}},
}

\bib{MR1623142}{article}{
      author={Hjelmborg, Jacob v.~B.},
      author={R{\o}rdam, Mikael},
       title={On stability of {$C^*$}-algebras},
        date={1998},
        ISSN={0022-1236},
     journal={J. Funct. Anal.},
      volume={155},
      number={1},
       pages={153\ndash 170},
         url={http://dx.doi.org/10.1006/jfan.1997.3221},
      review={\MR{1623142 (99g:46079)}},
}

\bib{MR925720}{article}{
      author={Hilsum, Michel},
      author={Skandalis, Georges},
       title={Morphismes {$K$}-orient\'es d'espaces de feuilles et
  fonctorialit\'e en th\'eorie de {K}asparov (d'apr\`es une conjecture d'{A}.
  {C}onnes)},
        date={1987},
        ISSN={0012-9593},
     journal={Ann. Sci. \'Ecole Norm. Sup. (4)},
      volume={20},
      number={3},
       pages={325\ndash 390},
         url={http://www.numdam.org/item?id=ASENS_1987_4_20_3_325_0},
      review={\MR{925720 (90a:58169)}},
}

\bib{MR625361}{article}{
      author={Julg, Pierre},
       title={{$K$}-th\'eorie \'equivariante et produits crois\'es},
        date={1981},
        ISSN={0249-6321},
     journal={C. R. Acad. Sci. Paris S\'er. I Math.},
      volume={292},
      number={13},
       pages={629\ndash 632},
      review={\MR{625361 (83b:46090)}},
}

\bib{MR2458205}{book}{
      author={Karoubi, Max},
       title={{$K$}-theory: {A}n introduction},
      series={Classics in Mathematics},
   publisher={Springer-Verlag, Berlin},
        date={2008},
        ISBN={978-3-540-79889-7},
         url={http://dx.doi.org/10.1007/978-3-540-79890-3},
      review={\MR{2458205 (2009i:19001)}},
}

\bib{MR2513335}{incollection}{
      author={Karoubi, Max},
       title={Twisted {$K$}-theory---old and new},
        date={2008},
   booktitle={{$K$}-theory and noncommutative geometry},
      series={EMS Ser. Congr. Rep.},
   publisher={Eur. Math. Soc., Z\"urich},
       pages={117\ndash 149},
         url={http://dx.doi.org/10.4171/060-1/5},
      review={\MR{2513335 (2010h:19010)}},
}

\bib{MR0238927}{article}{
      author={Karoubi, Max},
       title={Alg\`ebres de {C}lifford et {$K$}-th\'eorie},
        date={1968},
        ISSN={0012-9593},
     journal={Ann. Sci. \'Ecole Norm. Sup. (4)},
      volume={1},
       pages={161\ndash 270},
      review={\MR{0238927 (39 \#287)}},
}

\bib{MR582160}{article}{
      author={Kasparov, G.~G.},
       title={The operator {$K$}-functor and extensions of {$C^{\ast}
  $}-algebras},
        date={1980},
        ISSN={0373-2436},
     journal={Izv. Akad. Nauk SSSR Ser. Mat.},
      volume={44},
      number={3},
       pages={571\ndash 636, 719},
      review={\MR{582160 (81m:58075)}},
}

\bib{MR918241}{article}{
      author={Kasparov, G.~G.},
       title={Equivariant {$KK$}-theory and the {N}ovikov conjecture},
        date={1988},
        ISSN={0020-9910},
     journal={Invent. Math.},
      volume={91},
      number={1},
       pages={147\ndash 201},
         url={http://dx.doi.org/10.1007/BF01404917},
      review={\MR{918241 (88j:58123)}},
}

\bib{mathKT150906271}{misc}{
      author={Kellendonk, Johannes},
       title={{O}n the {C*}-algebraic approach to topological phases for
  insulators},
         how={preprint},
        date={2015},
        note={preprint, \href{http://arxiv.org/abs/1509.06271}{{\tt
  arXiv:1509.06271[math.KT]}}},
}

\bib{K2009}{article}{
      author={Kitaev, Alexei},
       title={Periodic table for topological insulators and superconductors},
        date={2009},
     journal={AIP Conference Proceedings},
      volume={1134},
      number={1},
}

\bib{MR1646047}{article}{
      author={Kumjian, Alexander},
      author={Muhly, Paul~S.},
      author={Renault, Jean~N.},
      author={Williams, Dana~P.},
       title={The {B}rauer group of a locally compact groupoid},
        date={1998},
        ISSN={0002-9327},
     journal={Amer. J. Math.},
      volume={120},
      number={5},
       pages={901\ndash 954},
  url={http://muse.jhu.edu/journals/american_journal_of_mathematics/v120/120.5kumjian.pdf},
      review={\MR{1646047 (2000b:46122)}},
}

\bib{MR2074181}{article}{
      author={Khoshkam, Mahmood},
      author={Skandalis, Georges},
       title={Crossed products of {$C^*$}-algebras by groupoids and inverse
  semigroups},
        date={2004},
        ISSN={0379-4024},
     journal={J. Operator Theory},
      volume={51},
      number={2},
       pages={255\ndash 279},
      review={\MR{2074181 (2005f:46122)}},
}

\bib{Kubota2}{misc}{
      author={Kubota, Yosuke},
       title={{{C}ontrolled topological phases and the bulk-edge
  correspondence}},
        date={2015},
        note={preprint \href{http://arxiv.org/abs/1511.05314}{{\tt
  arXiv:1511.05314[math-ph]}}},
}

\bib{MR1325694}{book}{
      author={Lance, E.~C.},
       title={Hilbert {$C^*$}-modules, {A} toolkit for operator algebraists},
      series={London Mathematical Society Lecture Note Series},
   publisher={Cambridge University Press, Cambridge},
        date={1995},
      volume={210},
        ISBN={0-521-47910-X},
         url={http://dx.doi.org/10.1017/CBO9780511526206},
      review={\MR{1325694 (96k:46100)}},
}

\bib{MR1686846}{article}{
      author={Le~Gall, Pierre-Yves},
       title={Th\'eorie de {K}asparov \'equivariante et groupo\"\i des. {I}},
        date={1999},
        ISSN={0920-3036},
     journal={$K$-Theory},
      volume={16},
      number={4},
       pages={361\ndash 390},
         url={http://dx.doi.org/10.1023/A:1007707525423},
      review={\MR{1686846 (2000f:19006)}},
}

\bib{MR3177819}{article}{
      author={Moutuou, El-ka{\"{\i}}oum~M.},
       title={Equivariant {$KK$}-theory for generalised actions and {T}hom
  isomorphism in groupoid twisted {$K$}-theory},
        date={2014},
        ISSN={1865-2433},
     journal={J. K-Theory},
      volume={13},
      number={1},
       pages={83\ndash 113},
         url={http://dx.doi.org/10.1017/is013010018jkt244},
      review={\MR{3177819}},
}

\bib{MR2917284}{article}{
      author={Paterson, Alan L.~T.},
       title={The stabilization theorem for proper groupoids},
        date={2012},
        ISSN={0362-1588},
     journal={Houston J. Math.},
      volume={38},
      number={1},
       pages={245\ndash 264},
      review={\MR{2917284}},
}

\bib{MR584266}{book}{
      author={Renault, Jean},
       title={A groupoid approach to {$C^{\ast} $}-algebras},
      series={Lecture Notes in Mathematics},
   publisher={Springer, Berlin},
        date={1980},
      volume={793},
        ISBN={3-540-09977-8},
      review={\MR{584266 (82h:46075)}},
}

\bib{MR1783408}{book}{
      author={R{\o}rdam, M.},
      author={Larsen, F.},
      author={Laustsen, N.},
       title={An introduction to {$K$}-theory for {$C^*$}-algebras},
      series={London Mathematical Society Student Texts},
   publisher={Cambridge University Press, Cambridge},
        date={2000},
      volume={49},
        ISBN={0-521-78334-8; 0-521-78944-3},
      review={\MR{1783408 (2001g:46001)}},
}

\bib{MR2082096}{article}{
      author={Roe, John},
       title={Paschke duality for real and graded {$C^*$}-algebras},
        date={2004},
        ISSN={0033-5606},
     journal={Q. J. Math.},
      volume={55},
      number={3},
       pages={325\ndash 331},
         url={http://dx.doi.org/10.1093/qjmath/55.3.325},
      review={\MR{2082096 (2005f:46108)}},
}

\bib{MR1018964}{article}{
      author={Rosenberg, Jonathan},
       title={Continuous-trace algebras from the bundle theoretic point of
  view},
        date={1989},
        ISSN={0263-6115},
     journal={J. Austral. Math. Soc. Ser. A},
      volume={47},
      number={3},
       pages={368\ndash 381},
      review={\MR{1018964 (91d:46090)}},
}

\bib{MR743845}{article}{
      author={Skandalis, Georges},
       title={Some remarks on {K}asparov theory},
        date={1984},
        ISSN={0022-1236},
     journal={J. Funct. Anal.},
      volume={56},
      number={3},
       pages={337\ndash 347},
         url={http://dx.doi.org/10.1016/0022-1236(84)90081-8},
      review={\MR{743845 (86c:46085)}},
}

\bib{mathph14067366}{article}{
      author={Thiang, Guo~Chuan},
       title={{O}n the {$K$}-theoretic classification of topological phases of
  matter},
        date={2015},
        ISSN={1424-0637},
     journal={Ann. Henri Poincar\'e},
       pages={1\ndash 38},
         url={http://dx.doi.org/10.1007/s00023-015-0418-9},
}

\bib{MR1775323}{article}{
      author={Trout, Jody},
       title={On graded {$K$}-theory, elliptic operators and the functional
  calculus},
        date={2000},
        ISSN={0019-2082},
     journal={Illinois J. Math.},
      volume={44},
      number={2},
       pages={294\ndash 309},
         url={http://projecteuclid.org/euclid.ijm/1255984842},
      review={\MR{1775323 (2002f:19009)}},
}

\bib{MR2231869}{article}{
      author={Tu, Jean-Louis},
       title={Groupoid cohomology and extensions},
        date={2006},
        ISSN={0002-9947},
     journal={Trans. Amer. Math. Soc.},
      volume={358},
      number={11},
       pages={4721\ndash 4747 (electronic)},
         url={http://dx.doi.org/10.1090/S0002-9947-06-03982-1},
      review={\MR{2231869 (2007i:22008)}},
}

\bib{MR1671260}{article}{
      author={Tu, Jean~Louis},
       title={La conjecture de {N}ovikov pour les feuilletages hyperboliques},
        date={1999},
        ISSN={0920-3036},
     journal={$K$-Theory},
      volume={16},
      number={2},
       pages={129\ndash 184},
         url={http://dx.doi.org/10.1023/A:1007756501903},
      review={\MR{1671260 (99m:46163)}},
}

\bib{MR2119241}{article}{
      author={Tu, Jean-Louis},
      author={Xu, Ping},
      author={Laurent-Gengoux, Camille},
       title={Twisted {$K$}-theory of differentiable stacks},
        date={2004},
        ISSN={0012-9593},
     journal={Ann. Sci. \'Ecole Norm. Sup. (4)},
      volume={37},
      number={6},
       pages={841\ndash 910},
         url={http://dx.doi.org/10.1016/j.ansens.2004.10.002},
      review={\MR{2119241 (2005k:58037)}},
}

\bib{MR947500}{article}{
      author={Van~Daele, A.},
       title={{$K$}-theory for graded {B}anach algebras. {I}},
        date={1988},
        ISSN={0033-5606},
     journal={Quart. J. Math. Oxford Ser. (2)},
      volume={39},
      number={154},
       pages={185\ndash 199},
         url={http://dx.doi.org/10.1093/qmath/39.2.185},
      review={\MR{947500 (89g:46113)}},
}

\bib{MR961241}{article}{
      author={Van~Daele, A.},
       title={{$K$}-theory for graded {B}anach algebras. {II}},
        date={1988},
        ISSN={0030-8730},
     journal={Pacific J. Math.},
      volume={134},
      number={2},
       pages={377\ndash 392},
         url={http://projecteuclid.org/euclid.pjm/1102689267},
      review={\MR{961241 (89k:46089)}},
}

\bib{MR994490}{article}{
      author={Van~Daele, A.},
       title={A note on the {$K$}-group of a graded {B}anach algebra},
        date={1988},
        ISSN={0037-9476},
     journal={Bull. Soc. Math. Belg. S\'er. B},
      volume={40},
      number={3},
       pages={353\ndash 359},
      review={\MR{994490 (90d:46102)}},
}

\bib{MR0106711}{book}{
      author={Wigner, Eugene~P.},
       title={Group theory and its application to the quantum mechanics of
  atomic spectra},
      series={Expanded and improved ed. Translated from the German by J. J.
  Griffin. Pure and Applied Physics. Vol. 5},
   publisher={Academic Press, New York-London},
        date={1959},
      review={\MR{0106711 (21 \#5442)}},
}

\end{biblist}
\end{bibdiv}
\end{document}